\documentclass[aip,graphicx]{revtex4-1}

\usepackage{amssymb}
\usepackage{amsthm}

\usepackage{amsmath}

\usepackage[all,cmtip]{xy}

\usepackage{float}

\DeclareMathOperator{\End}{End}

\newcommand{\mb}{\mathbb}

\newcommand{\Sp}{\mbox{Spin}}
\newcommand{\Spc}{\mbox{Spin}^{\mb{C}}}

\newtheorem{theorem}{Theorem}
\newtheorem{definition}{Definition}
\newtheorem{lemma}{Lemma}
\newtheorem{remark}{Remark}

\draft 

\begin{document}


\title{Immersion in $\mathbb{R}^n$ by Complex irreducible Spinors} 



\author{Rafael de Freitas Le\~ao}
\email[]{leao@ime.unicamp.br}
\affiliation{Institute of Mathematics, Statistics and Scientific Computation\\
13083-859 Campinas, SP, Brazil}

\author{Samuel Augusto Wainer}
\email[]{wainer@ita.br}
\affiliation{Technological Institute of Aeronautics \\
12228-900 Sao Jose dos Campos, SP, Brazil}


\date{\today}

\begin{abstract}

The first time that the connection between isometric immersion of surfaces and solutions of the Dirac equation appeared in the literature was in the seminal paper of Thomas Friedrich in $1998$. In consequence of that, several authors contributed to this topic hereafter, by obtaining the spinorial representation of $Spin$ manifolds with arbitrary dimension and also by presenting a generalization of the Weierstrass representation map, for example. All these results assume that the manifolds and bundles involved carry a $Spin$ structure, however this hypothesis is somehow restrictive, as for instance, if we consider complex manifolds, it is more natural to consider $Spin^\mathbb{C}$ structures. There exist an alternative to adapt this result to $Spin^\mathbb{C}$ manifolds, where the idea is to use the left regular representation of a Clifford algebra in itself to build the spinor bundles, but unlike the original works of Friedrich and Morel, this representation is not irreducible. Thus, this paper aims to present the spinorial representation of $Spin^\mathbb{C}$ manifolds into Euclidian space with arbitrary dimensions using spinors that came from an irreducible representation of a complex Clifford algebra.

\end{abstract}


\maketitle 

\section{Introduction}
The Weierstrass map is a classic method to use complex functions on the construction of minimal surfaces in the euclidean 3 space. On the seminal paper \cite{friedrich1998}, Thomas Friedrich shows that the Weierstrass map has relation with spinors and the Dirac equation. His idea was to consider an immersion $M^{2}\hookrightarrow \mathbb{R}^{3} $ of an oriented surface $M^{2}$ and fix a parallel spinor $\Phi $ on $\mathbb{R}^{3}.$ By restriction, a $ Spin $ structure on $\mathbb{R}^{3}$
canonically induces a $ Spin $ structure into $M^{2}$. Restricting $\Phi $ to $M^2,$ Friedrich produces a specific spinorial field $\varphi,$ with constant norm and that is solution of homogeneous Dirac equation $D(\varphi)=H\varphi.$ On the other hand, given a solution $ \varphi $ from Dirac equation, with constant norm, there is a symmetrical endomorphism $E:T(M^{2})\rightarrow
T(M^{2})$ such that the spinorial field satisfies a \textquotedblleft twistor equation\textquotedblright \  $ \nabla _{X}^{M^{2}}\varphi =E(X)\cdot \varphi ,$ 
where $\nabla ^{M^{2}}$ is the induced connection in the spinor bundle of $M^{2}.$ Friedrich shows that a solution of this twistor equation is equivalent to Gauss and Codazzi equations of isometric immersions. As a consequence, the solution $ \varphi $ of the Dirac equation $ D(\varphi )=H\varphi ,~\left\vert \varphi \right\vert =const>0 $ 
produces an isometric immersion of $M^{2}$ in $\mathbb{R}^{3}.$

The main result demonstrated by Friedrich in his paper is

\begin{theorem}
	Let $(M^{2},g)$ a $ 2 $ -dimensional oriented Riemannian manifold and $%
	H:M^{2}\rightarrow \mathbb{R}$ a smooth map. Then the following statements are equivalent:
	
	\begin{itemize}
		\item[(a)] There is an isometric immersion $(\tilde{M}
		^{2},g)\rightarrow \mathbb{R}^{3}$ of the universal covering $\tilde{M}^{2}$ of $M^{2}$ in $\mathbb{R}^{3}$ with mean curvature $H$.
		
		\item[(b)] There is a solution $\varphi $, with constant norm $\left\vert \varphi \right\vert =1,$ of Dirac equation $D\varphi
		=H\varphi $.
		
		\item[(c)] There is a pair $(\varphi ,E)$ consisting of a symmetric endomorphism $E$ such that\\ $tr(E)=-H$ and a spinorial field $\varphi$ such that $\nabla _{X}^{M^{2}}\varphi =E(X)\cdot \varphi $.
	\end{itemize}
\end{theorem}

Since Friedrich's work numerous works appeared, \cite{morel05,lawn08,lawnroth10,lawnroth11,bayard13,bayard13b,bayard17,vaz19}, showing how Dirac equations, spinors, Gauss-Codazzi equations and isometric immersions are related on manifolds of low dimension. In recent years, Bayard, Lawn and Roth \cite{bayard17}, using the left regular represantation of the Clifford Algebra on itself, generalized the spinorial Weierstrass map to manifolds of arbitrary dimension.

More precisely, let $M$ a $ p $-dimensional Riemannian manifold, $E\rightarrow M$ a real vector bundle of rank $q$ with metric and compatible connection. It is assumed that $ TM $ and $ E $ are oriented and $ Spin $, with $ Spin $ structures given by $P_{Spin_{p}}(TM)\longrightarrow P_{SO_{p}}(TM)\text{ and }P_{Spin_{q}}(E)\longrightarrow P_{SO_{q}}(E),$ where $P_{SO_{p}}(TM)$ and $P_{SO_{q}}(E)$ are the positively oriented frame bundles of $TM$ and $E$. It is considered the following $Spin_{p}\times Spin_{q}$-principal bundle over $M$ $P_{Spin_{p}\times Spin_{q}}:=P_{Spin_{p}}(TM)\times _{M}P_{Spin_{q}}(E),$ and the following associated fiber bundle $\Sigma :=P_{Spin_{p}\times Spin_{q}}\times _{\rho }Cl_{n},$ $ U\Sigma :=P_{Spin_{p}\times Spin_{q}}\times _{\rho }Spin_{n}\subset \Sigma ,$
where $\rho $ is the left regular representation and $n=p+q.$ Noting that the product $\left\langle \left\langle \cdot ,\cdot \right\rangle \right\rangle
:Cl_{n}\times Cl_{n} \rightarrow Cl_{n}, \ \left\langle \left\langle \xi ,\xi^\prime \right\rangle \right\rangle=\tau (\xi ^{\prime })\xi$, where $Cl_{n}$ is the $n$-dimensional real Clifford algebra and $\tau$ is the reversion on the algebra, is $Spin_{n}$-invariante, the authors considered de induced $Cl_{n}$-valued product $\left\langle \left\langle \cdot ,\cdot \right\rangle \right\rangle :\Sigma
\times \Sigma  \rightarrow Cl_{n},$ given by the expression  $\left\langle \left\langle \varphi ,\psi \right\rangle \right\rangle = \tau([\varphi^\prime ]) [\varphi] . $
where $\varphi =[p,[\varphi ]]\in \Gamma (\Sigma ),$ $\varphi ^{\prime
}=[p,[\varphi ^{\prime }]]\in \Gamma (\Sigma )$ and $p\in \Gamma
(P_{Spin_{p}\times Spin_{q}})$ is a spinorial frame.

Using these structures the main resut of Bayard, Lawn and Roth \cite{bayard17} can be stated as

\begin{theorem}
 Let $ M $ be a simply connected Riemannian $ p $-dimensional manifold, $E\rightarrow M$ a real vector bundle of rank $q$ with metric and compatible connection, suppose that $TM$ and $E$ are oriented and $ spin. $ Let $B:TM\times TM\rightarrow E$ a bilinear and symmetric form. Then the following statements are equivalent:
	
	\begin{enumerate}
		\item There exists a section $\varphi \in \Gamma (U\Sigma )$ such that 
		$$\nabla ^{\Sigma }\varphi =-\frac{1}{2}\sum_{j=1}^{p}e_{j}\cdot
		B(X,e_{j})\cdot \varphi , \ \ \forall \ X\in \Gamma (TM). $$
		
		\item There exists an isometric immersion $F:M\rightarrow \mathbb{R}^{n}$
		with normal bundle $E$ and second fundamental form $B$.
		
		Besides that, $dF=\xi,$ where $\xi $ is the $\mathbb{R}^{n}$-valued $1$-form defined by $$\xi (X)=\left\langle \left\langle X\cdot \varphi ,\varphi \right\rangle
		\right\rangle \in \mathbb{R}^n \subset Cl_n ,  \ \ \forall X\in \Gamma (TM).$$
		This expression generalizes the classic Weierstrass representation formula.
	\end{enumerate}
\end{theorem}

The assumption that the manifold carries a $\Sp$-structure is somewhat restrictive, for example, in the particular case of complex manifolds, is more natural to consider $\Spc$-structures. In the recent work of Leao and Wainer \cite{leaowainer2018}, it was demonstrated that the above solution can also be refined to $\Spc$-structures.

Looking for the development of this problem we can note that the initial works on low dimensions \cite{friedrich1998,morel05} usualy utilizes irreducible representations of the Clifford Algebras. On the other hand the generalizations \cite{bayard17,leaowainer2018} need an algebraic propertie of the left regular representation which is not an irreducible one. This difference with the fact that the Clifford Algebras are semi-simple, wich implies that every representation is completly reducible, raises the question if we can describe the Weierstrass map on higher dimensions using irreducible spinors.

In the present work we proved that it is possible to consider irreducible representations as long as we consider more then one solution of the Dirac equation. To show this, in the section \ref{Algebric Preliminaries} we present some algebric preliminaries on Clifford algebras ideals, $Spin$ and $Spin^{\mathbb{C}}$ groups; in the section \ref{512} we adapt the idea presented in \cite{bar98} to  $Spin^{\mathbb{C}}$ submanifolds relating the connections on the adapted $Spin^{\mathbb{C}}$ structures (\ref{relatingconnections}); in section \ref{idealspinors} we fix our notation, in order to take advantage of the multiplicative structure in the Clifford algebra and at the same time keep the $Spin$ representation comming from an irreducible representation, we build the spinor bundles from complex Clifford algebras ideals; in section \ref{hermitianproduct1} we define a $\mathbb{C}$-valued hermitian product using the Clifford algebra structure of the $\Spc$-Clifford bundle; finally in \ref{principalsection} we present our result Theorem \ref{principal} that gives a spinorial representation of $Spin^{\mathbb{C}}$ submanifolds in $\mathbb{R}^{n}$ using irreducible complex Clifford algebra spinors.

\section{Algebric Preliminaries} \label{Algebric Preliminaries}
Here we denote by $Cl_n$ the real Clifford Algebra on $\mathbb{R}^n$ and by $\mathbb{C}l_{n}=\mathbb{C}\otimes Cl_{n}$ its complexification. For each $n$, let $\nu _{n}$ the number of non-equivalent irreducible real representations of $Cl_{n}$ and denote by $\nu _{n}^{\mathbb{C}}$ the number of non-equivalent irreducible complex representations of $\mathbb{C}l_{n}$%
. Let $d_{n}=\dim _{\mathbb{R}}W$ where $W$ is an irreducible $\mathbb{R}$-module of $Cl_{n}$. Analogously be $d_{n}^{\mathbb{C}%
}=\dim _{\mathbb{R}}S$ where $S$ is an irreducible $\mathbb{C}$-module of $Cl_{n}$ and therefore also of $\mathbb{C}l_{n}$.

The following classification is well known \cite{lawson}:

\begin{theorem} 
	For $1\leq n\leq 8$ the values of $\nu _{n},\nu _{n}^{\mathbb{C}%
	},d_{n},d_{n}^{\mathbb{C}}$ are given by Table I.
	
	\begin{table}[h!] \label{classification}
		\begin{center}
			\begin{tabular}{|c||c|c|c||c|c|c|}
				\hline
				$n$ & $Cl_{n}$ & $\nu _{n}$ & $d_{n}$ & $\mathbb{C}l_{n}$ & $\nu _{n}^{%
					\mathbb{C}}$ & $d_{n}^{\mathbb{C}}$ \\ \hline\hline
				$1$ & $\mathbb{C}$ & $1$ & $2$ & $\mathbb{C\oplus C}$ & $2$ & $1$ \\ \hline
				$2$ & $\mathbb{H}$ & $1$ & $4$ & $\mathbb{C}(2)$ & $1$ & $2$ \\ \hline
				$3$ & $\mathbb{H}\oplus \mathbb{H}$ & $2$ & $4$ & $\mathbb{C}(2)\oplus 
				\mathbb{C}(2)$ & $2$ & $2$ \\ \hline
				$4$ & $\mathbb{C}(2)$ & $1$ & $8$ & $\mathbb{C}(4)$ & $1$ & $4$ \\ \hline
				$5$ & $\mathbb{C}(4)$ & $1$ & $8$ & $\mathbb{C}(4)\oplus \mathbb{C}(4)$ & $2$
				& $4$ \\ \hline
				$6$ & $\mathbb{R}(8)$ & $1$ & $8$ & $\mathbb{C}(8)$ & $1$ & $8$ \\ \hline
				$7$ & $\mathbb{R}(8)\oplus \mathbb{R}(8)$ & $2$ & $8$ & $\mathbb{C}(8)\oplus 
				\mathbb{C}(8)$ & $2$ & $8$ \\ \hline
				$8$ & $\mathbb{R}(16)$ & $1$ & $16$ & $\mathbb{C}(16)$ & $1$ & $16$ \\ 
				\hline
			\end{tabular}
			\caption{Values of $\nu _{n},\nu _{n}^{\mathbb{C}},d_{n},d_{n}^{\mathbb{C}%
				}.$}\label{tabela3}
		\end{center}
	\end{table}
	For $n>8$ it can be calculated using ($m,k\geq 1
	$):%
	\begin{eqnarray*}
	& Cl_{n+8} \simeq  Cl_n\otimes_{\mathbb{R}} \mathbb{R}(16), \ \ \mathbb{C}l_{n+2} \simeq \mathbb{C}l_n\otimes_{\mathbb{C}} \mathbb{C}(2),& \\
	&	\nu _{m+8k} =\nu _{m},\ \  \nu _{m+2k}^{\mathbb{C}}=\nu _{m}^{\mathbb{C}}, \ \ d_{m+8k} =2^{4k}d_{m}, \ \ d_{m+2k}^{\mathbb{C}}=2^{k}d_{m}^{\mathbb{C}}.&
	\end{eqnarray*}
\end{theorem}

\subsection{The $Spin^{\mathbb{C}}$ group}

\begin{definition} 
	The $Spin_n^{\mathbb{C}}$ group is defined by
	\begin{equation*}
	Spin_n^{\mathbb{C}} = \frac{Spin_n \times S^1}{ \{ (-1,-1) \} }, 
	\end{equation*}
	where $S^1=U(1) \subset \mathbb{C}$ denotes the complex unitary group.
\end{definition}

This is better understood within the complex Clifford algebra $\mathbb{C}l_n = \mathbb{C} \otimes Cl_n $. First note that $Spin_n$ and $S^1$ are subgroups of the group of invertible elements in $ \mathbb{C}l_n$ and $Spin_n \cap S^1 = \{ 1, -1 \}.$
Therefore the elements in $Spin_n^{\mathbb{C}}$ are equivalenceses classes in $Spin_n \times S^1$ by the relation $(p,s)\cong (-p,-s)$, obtained from $-1 \in \mathbb{C}l_n$. Then we have the following 
\begin{eqnarray*}
 Spin_n^{\mathbb{C}} &\hookrightarrow& Cl_n \otimes \mathbb{C} = \mathbb{C}l_n, \\
 \left[ P,s \right] &\mapsto& p  \otimes s.
\end{eqnarray*}

For the group $Spin_{n}^{\mathbb{C}}$ in this work we will define the following homomorphisms:
\begin{eqnarray*}
&\lambda_n:& Spin_n \rightarrow SO_n \text{ the double cover } \lambda_{n}(u)(v):=uvu^{-1}\\
&\lambda_n^{\mathbb{C}}:&Spin_{n}^{\mathbb{C}}\rightarrow SO_{n}, \text{defined by} \ \lambda_n^{\mathbb{C}}([p,s])=\lambda_n(p).\\
&i_n^{\mathbb{C}}:& Spin_{n} \rightarrow Spin_{n}^{\mathbb{C}} \text{ the natural inclusion, } i_n^{\mathbb{C}}(p)=[p,1]. \\
&j_n^{\mathbb{C}}:& S^{1}\rightarrow Spin_{n}^{\mathbb{C}} \text{ the natural inclusion, } j_n^{\mathbb{C}}(s)=[1,s].\\
&l_n^{\mathbb{C}}:& Spin_{n}^{\mathbb{C}}\rightarrow S^{1} \text{ defined by }	l_n^{\mathbb{C}}([p,s])=s^{2}. \\
&p_n^{\mathbb{C}}=&\lambda_n^{\mathbb{C}}\times l_n^{\mathbb{C}	}: Spin_{n}^{\mathbb{C}}\rightarrow SO_{n}\times S^{1} \text{ defined by } l_n^{\mathbb{C}}([p,s])=(\lambda_n(p),s^{2}).
\end{eqnarray*}

Then we have the following exact sequence for the $Spin_n^{\mathbb{C}}$ group:
\begin{equation*}
1 \longrightarrow \mathbb{Z}_2 \longrightarrow Spin_n^{\mathbb{C}} \overset{p_n^{\mathbb{C}}}{\longrightarrow} SO_n \times S^1 \longrightarrow 1.
\end{equation*}

\subsection{Spin representation}

Consider an irreducible representation of $\mathbb{C}l_{n}:=Cl_{n}\otimes \mathbb{C}$ 
\begin{equation*}
\bar{\rho}_{n}^{\mathbb{C}}:\mathbb{C}l_{n}\rightarrow \End_{\mathbb{C}}(S). 
\end{equation*}

Where $S$ is a $\mathbb{C}$-vectorial space such that $dim_{\mathbb{C}} S = 2^{\frac{n}{2}}$ if $n$ is even and $dim_{\mathbb{C}} S = 2^{\frac{n-1}{2}}$ if $n$ is odd.

Define $\rho _{n}^{\mathbb{C}}$ the complex $Spin$ representation as the restriction of $\bar{\rho}_{n}^{\mathbb{C}}$ given by the inclusions
\begin{eqnarray*} \label{phoo}
	Spin_{n} &\subset &Cl_{n}^{0}\subset Cl_{n}\subset \mathbb{C}l_{n}, \nonumber \\
	\rho _{n}^{\mathbb{C}} &:&=\left. \bar{\rho}_{n}^{\mathbb{C}}\right\vert
	_{Spin_{n}}:Spin_{n}\rightarrow \End_{\mathbb{C}}(S).
\end{eqnarray*}
We say that $S$ carries the $Spin_{n}$ representation.

	When $n$ is odd, the definition (\ref{phoo}) of $\rho _{n}^{\mathbb{C}%
	}$ is independent of which irreducible representation of $\mathbb{C}l_{n}$. Furthermore, when $n$ is odd, $\rho _{n}^{\mathbb{C}}$ is irreducible.
	
	When $n$ is even, there exist a decomposition
	\begin{equation*}
	\rho _{n}^{\mathbb{C}}=\left( \rho _{n}^{\mathbb{C}}\right) ^{+}\oplus
	\left( \rho _{n}^{\mathbb{C}}\right) ^{-}
	\end{equation*}
	as the sum of two non-equivalent irreducible complex representations of $Spin_{n}$.

In order to simplify the notation, where there is no likelihood of confusion, we simply denote $\bar{\rho}_{n}^{\mathbb{C}}= \rho_{n}^{\mathbb{C}} = \rho_{n}.$

\subsection{Complex Spin representation given by Clifford algebra ideals}  \label{sec15} \label{representacaoideal}

\subsubsection{Case $n$ even}

For the case $n=2k$ even, by the classification table \ref{classification},
there are isomorphisms
\begin{equation*}
\mathbf{i}_{2k}:\mathbb{C}l_{2k}\rightarrow \mathbb{C}(2^{k}).
\end{equation*}

Consider elements \textrm{f}$_{i}\in \mathbb{C}l_{2k},i=1,\cdots ,2^{k}$
such that
\begin{equation*}
\mathbf{i}_{2k}(\mathrm{f}_{i})=\overset{%
	\begin{array}{ccccc}
	&  & i \text{-th column} &  & 
	\end{array}%
}{\left( 
	\begin{array}{ccccc}
	0 & \cdots  & 0 & \cdots  & 0 \\ 
	\vdots  & \ddots  &  &  & \vdots  \\ 
	0 &  & 1 &  & 0 \\ 
	\vdots  &  &  & \ddots  & \vdots  \\ 
	0 & \cdots  & 0 & \cdots  & 0%
	\end{array}%
	\right) i\text{-th line},}.
\end{equation*}

The elements $\mathrm{f}_{i}, i=1, \cdots ,2^k$ are a complete set of primitive orthogonal idempotents:
\begin{eqnarray*}
& \mathrm{f}_{i}^{2}=\mathrm{f}_{i},\forall i, \mathrm{f}_{i}\mathrm{f}_{j}=0,i\neq j&,\\
& 1=f_1+ \cdots f_{2^k}.&
\end{eqnarray*}

Considering isomorphism with matrix algebra
it is obvious that the subsets
\begin{equation*}
I_{i}=\mathbb{C}l_{2k}\mathrm{f}_{i}=\{a\mathrm{f}_{i}:a\in \mathbb{C}%
l_{2k}\},i=1,\cdots ,2^{k},
\end{equation*}%
are left minimal ideals $\mathbb{C}l_{2k}$ with $\dim _{\mathbb{C}}I_{i}=2^{k},$ $I_{i}\cap I_{j}=\varnothing ,i\neq j,$
\begin{equation*}
\mathbf{i}_{2k}(I_{i})=\overset{%
	\begin{array}{ccccc}
	&  & i \text{-th} &  & 
	\end{array}%
}{\left( 
	\begin{array}{ccccc}
	0 & \cdots  & a_{1} & \cdots  & 0 \\ 
	\vdots  &  & \vdots  &  & \vdots  \\ 
	0 & \cdots  & a_{2^{k}} & \cdots  & 0%
	\end{array}%
	\right) }\text{ where }a_{j}\in \mathbb{C},j=1,\cdots ,2^{k},
\end{equation*}%
and it is also clear that $\mathrm{f}_{i}I_{i}\mathrm{f}_{i}\simeq \mathbb{%
	C}\mathrm{f}_{i}\simeq \mathbb{C}$.
	
From $1=\mathrm{f}_{1}+\cdots +\mathrm{f}_{2^{k}}$ the algebra $\mathbb{C}l_{2k}$ decomposes as the sum of minimal left ideals
\begin{equation*}
\mathbb{C}l_{2k}=\mathbb{C}l_{2k}\mathrm{f}_{1}\oplus \cdots \oplus \mathbb{C%
}l_{2k}\mathrm{f}_{2^{k}}=I_{1}\oplus \cdots \oplus I_{2^{k}}.
\end{equation*}

Considering the representation by left multiplication on the ideals $I_i$ 
\begin{eqnarray*}
\rho _{i} :\mathbb{C}l_{2k} \rightarrow \End_{\mathbb{C}}(I_{i})&,&i=1,\cdots
,2^{k}; \\
a \mapsto \rho _{i}(a):I_{i}&\rightarrow& I_{i}  \notag \\
\rho _{i}(a)\nu  &=&a\nu,   \notag
\end{eqnarray*}%
from the classification (table \ref{classification}) and the representation theory of matrix algebras follows immediately that the $\rho _{i}$'s are irreduticibles and equivalents.

\subsubsection{Case $n$ odd}

for the case $n=2k+1$ odd, by the classification table \ref{classification},
there are isomorphisms
\begin{equation*}
\mathbf{i}_{2k+1}:\mathbb{C}l_{2k+1}\rightarrow \mathbb{C}(2^{k})\oplus \mathbb{C}%
(2^{k}).
\end{equation*}

Consider the elements \textrm{f}$_{i;1},\mathrm{f}_{i;2}\in \mathbb{C}
l_{2k},i=1,\cdots ,2^{k}$ such that
\begin{eqnarray*}
\mathbf{i}_{2k+1}(\mathrm{f}_{i;1}) &=&%
i \text{-th}
\overset{%
	\begin{array}{ccccc}
	&  & i \text{-th} &  & 
	\end{array}%
}{\left( 
	\begin{array}{ccccc}
	0 & \cdots  & 0 & \cdots  & 0 \\ 
	\vdots  & \ddots  &  &  & \vdots  \\ 
	0 &  & 1 &  & 0 \\ 
	\vdots  &  &  & \ddots  & \vdots  \\ 
	0 & \cdots  & 0 & \cdots  & 0%
	\end{array}%
	\right) }\oplus \left( 
\begin{array}{ccccc}
0 & \cdots  & 0 & \cdots  & 0 \\ 
\vdots  & \ddots  &  &  & \vdots  \\ 
0 &  & 0 &  & 0 \\ 
\vdots  &  &  & \ddots  & \vdots  \\ 
0 & \cdots  & 0 & \cdots  & 0%
\end{array}%
\right) ,  \notag \\
\mathbf{i}_{2k+1}(\mathrm{f}_{i;2}) &=&\left( 
\begin{array}{ccccc}
0 & \cdots  & 0 & \cdots  & 0 \\ 
\vdots  & \ddots  &  &  & \vdots  \\ 
0 &  & 0 &  & 0 \\ 
\vdots  &  &  & \ddots  & \vdots  \\ 
0 & \cdots  & 0 & \cdots  & 0%
\end{array}%
\right) \oplus \overset{%
	\begin{array}{ccccc}
	&  & i \text{-th} &  & 
	\end{array}%
}{\left( 
	\begin{array}{ccccc}
	0 & \cdots  & 0 & \cdots  & 0 \\ 
	\vdots  & \ddots  &  &  & \vdots  \\ 
	0 &  & 1 &  & 0 \\ 
	\vdots  &  &  & \ddots  & \vdots  \\ 
	0 & \cdots  & 0 & \cdots  & 0%
	\end{array}%
	\right) }%
i \text{-th}
.
\end{eqnarray*}

The elements $\mathrm{f}_{i;\lambda},i=1,\cdots ,2^{k};\lambda=0,1$ are a complete set of primitive orthogonal idempotents: 
\begin{eqnarray*}
&\mathrm{f}_{i;\lambda}^{2} =\mathrm{f}_{i;\lambda},\forall i,l, \ \ \mathrm{f}_{i;1}\mathrm{f}_{j;2} =0,\forall i,j, \ \ \mathrm{f}_{i;\lambda}\mathrm{f}_{j;\lambda} =0,i\neq j,\lambda=0,1,&\\
&1=f_{1;1}+ \cdots f_{2^{k};1} + f_{1;2}+ \cdots f_{2^{k};2}&
\end{eqnarray*}

Considering the isomorphism with matrix algebra it is obvious that the subsets
\begin{equation*}
I_{i;\lambda}=\mathbb{C}l_{2k+1}\mathrm{f}_{i;\lambda}=\{a\mathrm{f}_{i;\lambda}:a\in \mathbb{C%
}l_{2k+1}\},i=1,\cdots ,2^{k};\lambda=0,1
\end{equation*}%
are left minimal ideals $\mathbb{C}l_{2k+1}$ with $\dim _{%
	\mathbb{C}}I_{i;\lambda}=2^{k},$  $I_{i;1}\cap I_{j;2}=\varnothing, \forall
i,j,$ with $~I_{i;1}\cap I_{j;2}=\varnothing , i\neq j,\lambda=0,1,$

 \begin{eqnarray*}
 \mathbf{i}_{2k+1}(I_{i;1}) &=&\left\{ \overset{%
	\begin{array}{ccccc}
	&  & i \text{-th} &  & 
	\end{array}%
}{\left( 
	\begin{array}{ccccc}
	0 & \cdots  & a_{1} & \cdots  & 0 \\ 
	\vdots  &  & \vdots  &  & \vdots  \\ 
	0 & \cdots  & a_{2^{k}} & \cdots  & 0%
	\end{array}%
	\right) }\oplus \left( 
\begin{array}{ccccc}
0 & \cdots  & 0 & \cdots  & 0 \\ 
\vdots  &  & \vdots  &  & \vdots  \\ 
0 & \cdots  & 0 & \cdots  & 0%
\end{array}%
\right)
\left. 
\right.
\right\} , \ \ \ \notag
\\
\mathbf{i}_{2k+1}(I_{i;2}) &=&\left\{ \left( 
\begin{array}{ccccc}
0 & \cdots  & 0 & \cdots  & 0 \\ 
\vdots  &  & \vdots  &  & \vdots  \\ 
0 & \cdots  & 0 & \cdots  & 0%
\end{array}%
\right) \oplus \overset{%
	\begin{array}{ccccc}
	&  & i \text{-th} &  & 
	\end{array}%
}{\left( 
	\begin{array}{ccccc}
	0 & \cdots  & a_{1} & \cdots  & 0 \\ 
	\vdots  &  & \vdots  &  & \vdots  \\ 
	0 & \cdots  & a_{2^{k}} & \cdots  & 0%
	\end{array}%
	\right) }
\left. 
\right. 
\right\}, \ \ \
\end{eqnarray*}
where $a_{j}\in \mathbb{C}, j=1,\cdots ,2^{k}.$ It is also clear that $\mathrm{f}_{i;\lambda}I_{i}\mathrm{f}_{i;\lambda}\simeq 
\mathbb{C}\mathrm{f}_{i;\lambda}\simeq \mathbb{C}$. 

From
\begin{equation*}
 1=\left(\mathrm{f}_{1;1}+\cdots+\mathrm{f}_{2^{k};1} \right) + \left( \mathrm{f}_{1;2}+\cdots +\mathrm{f}_{2^{k};2} \right),
\end{equation*} 
the algebra $\mathbb{C}l_{2k+1}$ decomposes as the sum of minimal left ideals
\begin{eqnarray*}
\mathbb{C}l_{2k+1} &=& \left( \mathbb{C}l_{2k+1}\mathrm{f}_{1;1}\oplus \cdots \oplus \mathbb{C}l_{2k+1}%
\mathrm{f}_{2^{k};1}\right) \oplus \left( \mathbb{C}l_{2k+1}\mathrm{f}_{1;2}\oplus \cdots
\oplus \mathbb{C}l_{2k+1}\mathrm{f}_{2^{k};2} \right)  \notag \\
&=&\left( I_{1;1}\oplus \cdots \oplus I_{2^{k};1}\right) \oplus \left( I_{1;2}\oplus \cdots \oplus
I_{2^{k};2} \right)
\end{eqnarray*}

Considering the representation by left multiplication on the ideals $I_{i;\lambda}$
\begin{eqnarray*}
\rho _{i;\lambda} :\mathbb{C}l_{2k} &\rightarrow& \End_{\mathbb{C}%
}(I_{i;\lambda}),i=1,\cdots ,2^{k},\lambda=0,1;  \notag \\
a &\mapsto &\rho _{i;\lambda}(a):I_{i;\lambda}\rightarrow I_{i;\lambda}  \notag \\
&& \ \ \ \ \ \ \rho _{i;\lambda}(a)\nu  =a\nu 
\end{eqnarray*}%
from the classification (table \ref{classification}) and the representation theory of matrix algebras follows immediately that there are two irreducible representations classes and
\begin{equation*}
\rho _{1;1} \simeq \cdots \simeq \rho _{2^{k};1}, \ \ \ \
\rho _{1;2} \simeq \cdots \simeq \rho _{2^{k};2}.
\end{equation*}

\section{Adapted $Spin^{\mathbb{C}}$ Structure} \label{512}

In this work whenever we use the noun submanifold we are referring to the concept of immersed submanifold. Remember that a manifold $M$ is immersed submanifold of $N$ if there is an injective smooth map $F:M \rightarrow N$ with injective derivative $dF:TM \rightarrow TN$ (\cite{warner} 22).

It will be relevant for our study to consider the following inclusion $n=p+q$:
\begin{eqnarray*}
	SO(p)\times SO(q) &\subset &SO(p+q) \notag \\
	(\lambda _{p}(u),\lambda _{q}(w)) &:&\mathbb{R}^{p}\times \mathbb{R}
	^{q}\rightarrow \mathbb{R}^{p}\times \mathbb{R}^{q} \notag \\
	(\lambda _{p}(u),\lambda _{q}(w))(v_{1},v_{2}) &=&(\lambda_{p}(u)v_{1},\lambda _{q}(w)v_{2}).
\end{eqnarray*}

\begin{definition}
	We define the adapted $ Spin $ group as
	\begin{equation*}
	\mathcal{S}_n:=\{uv;u\in Spin_p,v\in Spin_q\}\subset Spin_n 
	\end{equation*}
	and note that $\mathcal{S}_n=\lambda _{n}^{-1}\left( SO(p)\times SO(q)\right)
	$ and $\mathcal{S}_n \simeq \frac{Spin(p) \times Spin(q)}{(1,1),(-1,-1)}.$
\end{definition}

Note that $\left. \lambda _{n}\right\vert _{\mathcal{S}_n}:\mathcal{S}\rightarrow
SO(p)\times SO(q)$ is a double cover map
\begin{eqnarray*}
	\left. \lambda _{n}\right\vert _{\mathcal{S}_n}(uw)(v_{1},v_{2}) &=&\lambda
	_{n}(uw)(v_{1}+v_{2})= \notag \\
	uw(v_{1}+v_{2})w^{-1}u^{-1}
	&=&uv_{1}u^{-1}+wv_{2}w^{-1}=(uv_{1}u^{-1},wv_{2}w^{-1}),
\end{eqnarray*}
identify $(v_1,v_2) \in \mathbb{R}^p \times \mathbb{R}^q$ with $v_1 + v_2 \in \mathbb{R}^{n}$.

Here we adapt the idea presented in \cite{bar98} to $Spin^{\mathbb{C}}$ submanifolds. Let $Q$ an $n$-dimensional riemannian $Spin^{\mathbb{C}}$ manifold and $M\hookrightarrow Q$ an $p$-dimensional $Spin^{\mathbb{C}}$ submanifold. Put in $M$ the induced metric from $Q$. Consider $P_{SO(n)}$ the bundle of positively oriented frames of $Q$ and $P_{SO(p)}$ the bundle of positively oriented frames of $M.$ $P_{S^{1}}$ is the $S^{1}$-principal bundle associated with the $Spin^{\mathbb{C}}$ structure of $Q$
and $P_{S^{1}}^{1}$ is the $S^{1}$-principal bundle associated with the $Spin^{\mathbb{C}}$ of $M$.

Denote by $\left. P_{SO(n)}\right\vert _{M}$ the frame bundle $Q $ restricted to $M$ with structure group $SO(p)\times SO(q).$ The same with $\left. P_{S^{1}}\right\vert _{M}$ with structure group $S^{1}.$

Let $e_{1},...,e_{p}$ a local positively oriented base of tangent bundle of $M$, and $f_{1},...,f_{q}$ a local positively oriented base of normal bundle $E$. Fix $$h=h_{1}\oplus h_{2}=(e_{1},...,e_{p},f_{1},...,f_{q}):M\rightarrow \left. P_{SO(n)}\right\vert _{M}$$ a local section of frame bundle of $Q$ restricted $M$ and $l:M\rightarrow \left. P_{S^{1}}\right\vert
_{M} $ a local section of $S^{1}$-principal bundle restricted to $M$. 

Let
\begin{equation*}
\Lambda ^{\mathbb{C}Q}:P_{Spin_{n}^{\mathbb{C}}}\rightarrow
P_{SO(n)}\times P_{S^{1}}, 
\end{equation*}
the $Spin^{\mathbb{C}}$ structure of $Q$ and
\begin{equation*}
\Lambda^{1\mathbb{C}}:P_{Spin_{p}^{\mathbb{C}}}\rightarrow
P_{SO(p)}\times P_{S^{1}}^{1}, 
\end{equation*}
the $Spin^{\mathbb{C}}$ structure of $M$.

From the natural inclusion $\left. P_{SO(n)}\right\vert _{M}\subset
P_{SO(n)},$ define the $\mathcal{S}\times S^{1}$-principal bundle:
\begin{equation*}
\left. P_{Spin^{\mathbb{C}}_n}\right\vert _{M}:=\left( \Lambda ^{\mathbb{C}Q}\right) ^{-1}\left( \left. P_{SO(n)}\right\vert _{M}\times \left.
P_{S^{1}}\right\vert _{M}\right) . 
\end{equation*}
If we denote the transition functions of $\left. P_{Spin^{\mathbb{C}}_n}\right\vert _{M}$ by $\tilde{g}_{\alpha \beta} = \left[ h_{\alpha \beta} , z_{\alpha \beta} \right] \in Spin_n^{\mathbb{C}}$ e and the transition functions of $P_{Spin_{(n)}^{\mathbb{C}}}$ by $\tilde{g}_{\alpha \beta}^1 = \left[ h_{\alpha \beta}^1 , z_{\alpha \beta}^1 \right] \in Spin_n^{\mathbb{C}}$ it is not difficult to define a $Spin^{\mathbb{C}}$ structure in $E$
\begin{equation*}
\Lambda^{2\mathbb{C}Q}:P_{Spin_{(m)}^{\mathbb{C}}}\rightarrow
P_{SO(m)}\times P_{S^{1}}^{2},  
\end{equation*}
where the bundle $P_{Spin_{(m)}^{\mathbb{C}}}$ is such that the transiction funcitons $\tilde{g}_{\alpha \beta}^2 = \left[ h_{\alpha \beta}^2 , z_{\alpha \beta}^2 \right]$ satisfy $\tilde{g}_{\alpha \beta}^1 \tilde{g}_{\alpha \beta}^2 = \tilde{g}_{\alpha \beta}$.

The transition functions that define $\left.
P_{S^{1}}\right\vert _{M}$ 
are the product of the transitions functions of $P_{S^{1}}^{1}$ e $P_{S^{1}}^{2}$, there is a morphism
canonical $\Phi :P_{S^{1}}^{1}\times _{M}P_{S^{1}}^{2}\rightarrow \left.
P_{S^{1}} \right|_M$ such that $\Phi (p_{1}\cdot s_{1},p_{2}\cdot s_{2})=\Phi
(p_{1},p_{2})s_{1}s_{2},$ $p_{1}\in P_{S^{1}}^{1},$ $p_{2}\in P_{S^{1}}^{2},$
$s_{1},s_{2}\in S^{1},$ which in a local trivialization makes the following diagram commute:
\begin{center}
	\begin{equation*}
	\xymatrixcolsep{1pc}\xymatrixrowsep{1pc}\xymatrix{
	P_{S^{1}}^{1}\times _{M}P_{S^{1}}^{2} \ar[rr]^{\ \Phi} \ar[dd] & & \left. P_{S^{1}} \right|_M \ar[dd] \\	& \\
	U_{\alpha} \times S^1 \times S^1 \ar[rr]^{\phi_\alpha} & & U_{\alpha} \times S^1	}
	\end{equation*}
\end{center}
where $\phi_\alpha(x,r,s)=(x,rs).$

\subsection{Adapted connections}

Admit the following connection $ 1 $-form:

\begin{equation*}
w^{\mathbb{C}Q}=w^{Q}\oplus iA:T(P_{SO(n)}\times P_{S^{1}})\rightarrow 
\mathfrak{so}(n)\oplus i\mathbb{R}, 
\end{equation*}%
where $w^{Q}:T(P_{SO(n)})\rightarrow \mathfrak{so}(n)$ is the Levi-Civita connection of $P_{SO(n)}$ and $iA:TP_{S^{1}}\rightarrow i\mathbb{R}$ is an arbitrary connection on $P_{S^{1}}.$

The connection $1$-form on $\left. P_{SO(n)}\right\vert _{M}\times
\left. P_{S^{1}}\right\vert _{M}$ will be defined by
\begin{eqnarray*}
	&w^{\mathbb{C}ad} :T\left( \left. P_{SO(n)}\right\vert _{M}\times \left.
	P_{S^{1}}\right\vert _{M}\right) \rightarrow \left( \mathfrak{so}(p)\oplus 
	\mathfrak{so}(q)\right) \oplus i\mathbb{R},& \notag \\
	&w^{\mathbb{C}ad}(dh(X)\oplus dl(X)) =(w^{M}\oplus w^{\bot })(dh(X))\oplus
	iA(dl(X)),&
\end{eqnarray*}%
where $w^{M}:TP_{SO(n)}\rightarrow \mathfrak{so}(n)$\ is the Levi-Civita connection of $P_{SO(n)}$, $w^{\bot }:TP_{SO(m)}\rightarrow \mathfrak{so}%
(m) $ and is normal the connection. For $p\in M$ and $X\in T_{p}M$ Gauss's formula tells us that, with respect to decomposition $%
T_{p}Q=T_{p}M\oplus E_{p}$,%
\begin{eqnarray*}
\nabla _{X}^{Q} &=&\left( 
\begin{array}{cc}
\nabla _{X}^{M} & -B(X,~)^{\ast } \notag \\ 
B(X,~) & \nabla _{X}^{\bot }%
\end{array}%
\right) , \\
\nabla _{X}^{Q}-\nabla _{X}^{M}\oplus \nabla _{X}^{\bot } &=&\left( 
\begin{array}{cc}
0 & -B(X,~)^{\ast } \\ 
B(X,~) & 0%
\end{array}%
\right) .
\end{eqnarray*}
That in matrix form can be written as:%
\begin{eqnarray}
&&w^{\mathbb{C}Q}(dh(X)\oplus dl(X))-w^{\mathbb{C}ad}(dh(X)\oplus dl(X))
\label{diference} \notag \\
&=&w^{Q}(dh(X))\oplus iA(dl(X))-(w^{M}\oplus w^{\bot })(dh(X))\oplus
iA(dl(X)) \notag \\
&=&\left( 
\begin{array}{cc}
0 & -\left\langle B(X,e_{i}),f_{j}\right\rangle _{j,i} \\ 
\left\langle B(X,e_{i}),f_{j}\right\rangle _{i,j} & 0%
\end{array}
\right) \oplus 0.
\end{eqnarray}

Consider $w^{Spin^{\mathbb{C}}Q}:TP_{Spin^{\mathbb{C}}(n)}\rightarrow 
\mathfrak{spin}(n)\oplus i\mathbb{R}$ the lifted 1-form to $\mathfrak{%
	spin}(n)\oplus i\mathbb{R}$, by the isomorphism $\mathfrak{spin}(n)\oplus i%
\mathbb{R\simeq }\mathfrak{so}(n)\oplus i\mathbb{R}$; and $$w^{Spin^{\mathbb{C%
	}}ad}:T\left. P_{Spin^{\mathbb{C}}(n)}\right\vert _{M}\rightarrow 
\mathfrak{spin}(p)\oplus \mathfrak{spin}(q)\oplus i\mathbb{R}\subset 
\mathfrak{spin}(n)\oplus i\mathbb{R}$$ the lift 1-form to $\mathfrak{spin}(p)\oplus \mathfrak{spin}(q)\oplus i\mathbb{R}\subset \mathfrak{spin}(n)\oplus i\mathbb{R}$:%
\begin{eqnarray*}
\lambda _{m\ast }^{\mathbb{C}}\left( w^{Spin^{\mathbb{C}}Q}(dh(X)\oplus
dl(X))\right) &=&w^{\mathbb{C}Q}(dh(X)\oplus dl(X)); \\
\lambda _{m\ast }^{\mathbb{C}}\left( w^{Spin^{\mathbb{C}}ad}(dh(X)\oplus
dl(X))\right) &=&w^{\mathbb{C}ad}(dh(X)\oplus dl(X)).
\end{eqnarray*}

From Eq.(\ref{diference}) we will have:%
\begin{eqnarray} \label{relatingconnections}
&&\lambda _{m\ast }^{\mathbb{C}}\left( w^{Spin^{\mathbb{C}}Q}(dh(X)\oplus
dl(X))\right) -\lambda _{m\ast }^{\mathbb{C}}\left( w^{Spin^{\mathbb{C}%
	}ad}(dh(X)\oplus dl(X))\right) \notag \\
&=&\left( 
\begin{array}{cc}
0 & -\left\langle B(X,e_{i}),f_{j}\right\rangle _{j,i} \\ 
\left\langle B(X,e_{i}),f_{j}\right\rangle _{i,j} & 0%
\end{array}%
\right) \oplus 0, \notag \\
&&\left( w^{Spin^{\mathbb{C}}Q}(dh(X)\oplus dl(X))\right) -w^{Spin^{\mathbb{C%
	}}ad}(dh(X)\oplus dl(X))  \label{2} \notag \\
&=&\frac{1}{2}\sum^p_{i=1}\sum^q_{j=1}\left\langle B(X,e_{i}),f_{j}\right\rangle
e_{i}\cdot f_{j}\oplus 0.
\end{eqnarray}

\section{Clifford algebra ideal spinors} \label{idealspinors}

Consider irreducible representation by left multiplication on the minimal ideals of $\mathbb{C}l_{n},$ according with notation established in Section \ref{representacaoideal}:
\begin{eqnarray*}
\rho _{i}^{n} &:&\mathbb{C}l_{2k}\rightarrow \End_{\mathbb{C}}\left(\mathbb{C}l_{2k}\mathrm{f}_{i}\right)=\End_{\mathbb{C}}\left(I_{i}^{(2k)}\right),i=1,\cdots ,2^{k}, 
\notag \\
&& a \mapsto \rho _{i}^{n}(a):\beta \mathrm{f}_{i}\mapsto a\beta \mathrm{f}%
_{i},~~~\text{for the case }n=2k\text{ even.}  \notag \\
\rho _{i;\lambda}^{n} &:&\mathbb{C}l_{2k+1}\rightarrow \End_{\mathbb{C}}\left(\mathbb{C}l_{(2k+1)}\mathrm{f}_{i;\lambda}\right)=\End_{\mathbb{C}}\left(I_{i;\lambda}^{(2k+1)}\right),i=1,\cdots
,2^{k},\lambda=0,1,  \notag \\
&& a \mapsto \rho _{i;\lambda}^{n}(a):\beta \mathrm{f}_{i;\lambda}\mapsto a\beta \mathrm{f%
}_{i;\lambda},~~~\text{for the case }n=2k+1\text{ odd,}
\end{eqnarray*}

Their restrictions to $Spin_{n}^{\mathbb{C}}$ will also be denoted by $\rho^n _{i}$ or $\rho^n _{i;\lambda}$. Remember that the
representations $\rho^n _{1}\simeq \cdots \simeq \rho^n _{2^{k}}$ are equivalent, as well as $\rho^n _{1;\lambda}\simeq \cdots \simeq \rho^n _{2^{k};\lambda}$, $%
\lambda=0,1$.

In what follows, always considering the parities of $m$ and $n$, we suppress indexes $ i, j, l $ when there is no risk of confusion. A more detailed description of the bundles below can be found in the \ref{apendicea}.

Given these irreducible representations, we define the following complex spinors bundles:
\begin{eqnarray}
\sum\nolimits^{\mathbb{C}}M &:&=P_{Spin_{p}^{\mathbb{C}}}\times _{\rho
	^{p}}I^{p},~\sum\nolimits^{\mathbb{C}}E:=P_{Spin_{q}^{\mathbb{C}}}\times
_{\rho ^{q}}I^{q},  \notag \\
\sum\nolimits^{\mathbb{C}}Q &:&=P_{Spin_{n}^{\mathbb{C}}}\times _{\rho
	^{n}}I^{n},~\left. \sum\nolimits^{\mathbb{C}}Q\right\vert
_{M}:=\left. P_{Spin_{n}^{\mathbb{C}}}\right\vert _{M}\times _{\rho
	^{n}}I^{n}.  \notag
\end{eqnarray}

From \cite[pp. 5]{bar98} we can compare the spinor modules of the Clifford algebra of a direct sum with the spinor modules associated with each factor. Besides that, using the fact that the transition functions of $\left. P_{Spin^{\mathbb{C}	}_n}\right\vert _{M}$ are the product of the transition functions of $P_{Spin^{\mathbb{C}}_p}$ and $P_{Spin^{\mathbb{C}}_q}$ , it's not difficult to get:
\begin{eqnarray*}
&& \text{For the case that } p \text{ and } q \text{ are not both odd:} \notag\\
&& \sum\nolimits^{\mathbb{C}} :=\sum\nolimits^{\mathbb{C}}M\otimes \sum\nolimits^{\mathbb{C}}E \simeq \left. \sum\nolimits^{\mathbb{C}}Q\right\vert _{M}. \notag \\
&& \text{For the case that } p \text{ and } q \text{ are both odd:} \notag\\
&& \sum\nolimits^{\mathbb{C}} :=\left( \sum\nolimits^{\mathbb{C}}M\otimes \sum\nolimits^{\mathbb{C}}E \right) \oplus \left( \sum\nolimits^{\mathbb{C}}M\otimes \sum\nolimits^{\mathbb{C}}E \right)  \simeq \left. \sum\nolimits^{\mathbb{C}}Q\right\vert _{M}.
\end{eqnarray*}

Fix $\nabla ^{\Sigma ^{\mathbb{C}}Q},\nabla ^{\Sigma ^{\mathbb{C}}M}$ and $%
\nabla ^{\Sigma ^{\mathbb{C}}E}$ the Levi-Civita connections on $\sum^{\mathbb{C}%
}Q,\sum^{\mathbb{C}}M$ and $\sum^{\mathbb{C}}E$ respectively. The connection on $\sum\nolimits^{\mathbb{C}}$ will be given by

\begin{equation*}
\nabla ^{\Sigma ^{\mathbb{C}}}:=\nabla ^{\Sigma ^{\mathbb{C}}M}\otimes
Id+Id\otimes \nabla ^{\Sigma ^{\mathbb{C}}E},
\end{equation*}
from which, using eq.~\eqref{2}, it follows the spinorial Gauss formula:

\begin{equation}
\nabla _{X}^{\Sigma ^{\mathbb{C}}Q}-\nabla _{X}^{\Sigma ^{\mathbb{C}}}=%
\frac{1}{2}\sum_{i=1}^{p}\sum_{j=1}^{q}\left\langle B(X,e_{i}),f_{j}\right\rangle
e_{i}\cdot f_{j}=\frac{1}{2}\sum_{i=1}^{p}e_{i}\cdot B(X,e_{i})\cdot
\label{gauss2}
\end{equation}

\section{A $\mathbb{C}$-valued hermitian product on the spinor bundle} \label{hermitianproduct1}

Let's define the following antiautomorphism
\begin{eqnarray*}
\tau &:&\mathbb{C}l_{n}\rightarrow \mathbb{C}l_{n}  \notag \\
\tau (a~e_{i_{1}}e_{i_{2}}\cdots e_{i_{k}}) &:&=(-1)^{k}\bar{a}%
~e_{i_{k}}\cdots e_{i_{2}}e_{i_{1}},
\end{eqnarray*}
where $\{e_{1},\cdots e_{n}\}$ is an orthonormal basis of $\mathbb{R}^{n}\subset \mathbb{C}l_{n}$ and $a\in \mathbb{C}.$ For simplicity, we can
write $\tau (\xi )=\bar{\xi},$ $\xi \in \mathbb{C}l_{n}.$

\begin{lemma}
When we consider the isomorphism (sec.~\ref{sec15}) of $\mathbb{C}l_{n}$ with $\mathbb{C}(2^{\frac{n}{2}})$ or $\mathbb{C}(2^{\frac{n-1}{2}})\oplus \mathbb{C}(2^{\frac{n-1}{2}})$ the antiautomorphism $\tau $ is translated as the conjugate transpose in matrix algebra.
\end{lemma}

\begin{proof}
	If $n=2k$ is even, consider the isomorphism $i_{2k}:\mathbb{C}l_{2k}\rightarrow \mathbb{C}(2^{k})$ and the representation by left multiplication
	\begin{equation*}
	\mathbb{C}(2^{k}) \rightarrow  \End_{\mathbb{C}}(\mathbb{C}^{2^{k}}). \notag
	\\
	\end{equation*}
	We know \cite[pp. 24]{friedrich00}  that there is in $\mathbb{C}(2^{k})$ a hermetian product $\left\langle \cdot ,\cdot \right\rangle $ such that
	\begin{equation*}
	\left\langle i_{2k}(v)A,B\right\rangle  =-\left\langle
	A,i_{2k}(v)B\right\rangle ,  \ \ \ \
	\forall A,B \in \mathbb{C}(2^{k});v\in \mathbb{R}^{n}\subset \mathbb{C}%
	l_{n},
	\end{equation*}
	but $\tau (v)=-v,$ therefore
	\begin{equation*}
	\left\langle i_{2k}(v)A,B\right\rangle  =\left\langle A,i_{2k}(\tau
	(v))B\right\rangle ,  \ \ \ \
	\forall A,B \in \mathbb{C}^(2^{k});v\in \mathbb{R}^{n}\subset \mathbb{C}%
	l_{n}.
	\end{equation*}%
	Thus $,i_{2k}(\tau (v))$ is the adjoint operator $i_{2k}(v)$ with respect to the product $\left\langle \cdot ,\cdot \right\rangle $. Choosing a convenient base, without loss of generality we have that $i_{2k}(\tau (v))$ is the conjugate transposed matrix of $i_{2k}(v).$ Since this is valid for vectors then it is valid for all $\varphi \in \mathbb{C}l_n$, i.e.
	\begin{equation*}
	i_{2k}(\tau (\varphi )) = \left( i_{2k}(\varphi )\right) ^{\ast },  \ \ \ \
	\forall \varphi  \in \mathbb{C}l_{n}.
	\end{equation*}
	That is, according to isomorphism $i_{2k}$, the antiautomorphism $\tau $ is translated as the conjugate transpose in matrix algebra.
	
	Note that the same is true if $n=2k+1$ odd, it is sufficient to consider the two non-equivalent natural representations $\mathbb{C}(2^{k})\oplus \mathbb{C}(2^{k})\rightarrow \End_{\mathbb{C}}(\mathbb{C}^{2^{k}}).$
\end{proof}

Finally we can present the following definition:

\begin{definition}
	We have the following $\mathbb{C}l_{n}$-valued hermitian product 
	\begin{eqnarray*}
	\left\langle \left\langle \cdot ,\cdot \right\rangle \right\rangle &:&%
	\mathbb{C}l_{n}\times \mathbb{C}l_{n}\rightarrow \mathbb{C}l_{n}  \notag \\
	(\xi _{1},\xi _{2}) &\mapsto &\left\langle \left\langle \xi _{1},\xi
	_{2}\right\rangle \right\rangle =\tau (\xi _{2})\xi _{1}.  \label{product2}
	\end{eqnarray*}
\end{definition}

\begin{remark} Note that the following statements are valid:
	\begin{enumerate}
		\item $\left\langle \left\langle \cdot ,\cdot \right\rangle
		\right\rangle $ is $Spin_{n}^{\mathbb{C}}$-invariant:%
		\begin{equation*}
		\left\langle \left\langle (g\otimes s)\xi _{1},(g\otimes s)\xi
		_{2}\right\rangle \right\rangle  =s\overline{s}\tau (\xi _{2})\tau (g)g\xi
		_{1}=\tau (\xi _{2})\xi _{1}=\left\langle \left\langle \xi _{1},\xi
		_{2}\right\rangle \right\rangle ,  
        ( g\otimes s ) \in Spin_{n}^{\mathbb{C}}\subset \mathbb{C}l_{n},
		\end{equation*}
		since $Spin_{n}\subset \{g\in Cl_{n}^{0};\bar{g}g=1\}$ and $s\in S^{1}\subset 
		\mathbb{C}$.
		
		\item The eq.~\eqref{product2} induces the following $\mathbb{C}$-valued map 
		\begin{gather*}
		\sum\nolimits^{\mathbb{C}}Q\times \sum\nolimits^{\mathbb{C}}Q\rightarrow 
		\mathbb{C}  \notag \\
		(\varphi _{1},\varphi _{2})=([P,[\varphi _{1}]],[P,[\varphi _{2}]])\mapsto
		\left\langle \left\langle \lbrack \varphi _{1}],[\varphi _{2}]\right\rangle
		\right\rangle =\tau ([\varphi _{2}])[\varphi _{1}],
		\end{gather*}
		where $[\varphi _{1}],[\varphi _{2}]\in I^{n}=\mathbb{C}l_{n}\mathrm{f}$ are the representatives of $\varphi _{1},\varphi _{2}$ in a given $Spin^{\mathbb{C}}_n$-frame $P\in \Gamma \left( P_{Spin^{\mathbb{C}%
			}_n}\right) .$
		
		 $I^{n}=\mathbb{C}l_{n}\mathrm{f}$ is minimal ideal, with $\mathrm{f}$ a primitive idempotent. 
		
		Note that $\left\langle \left\langle \lbrack \varphi _{1}],[\varphi _{2}]\right\rangle
		\right\rangle = \tau ([\varphi
		_{2}])[\varphi _{1}]\in \tau (\mathrm{f})\mathbb{C}l_{(n+m)}\mathrm{f}=%
		\mathrm{f}\mathbb{C}l_{(n+m)}\mathrm{f}\simeq \mathbb{C}.$
	\end{enumerate}
\end{remark}

\begin{lemma}
The connection $\nabla ^{\Sigma ^{\mathbb{C}}Q}$ is compatible with the product $\left\langle \left\langle \cdot ,\cdot \right\rangle \right\rangle $
\end{lemma}
\begin{proof}
	Fix $s=(e_{1},...,e_{n}):U\subset M\subset Q\rightarrow P_{SO(n+m)}$
	a local section of the frame bundle, $l:U\subset M\subset
	Q\rightarrow P_{s^{1}}$ and a local section of the $S^{1}$-principal bundle , $w^{Q}:T(P_{SO(n+m)})\rightarrow so(n+m)$ the Levi-Civita connection of $P_{SO(n+m)}$ and $iA:TP_{S^{1}}\rightarrow i\mathbb{R}$ an arbitrary connection in $P_{S^{1}}$, denote by $w^{Q}(ds(X))=(w_{ij}(X))\in so(n+m),$ $iA(dl(X))=iA^{l}(X).$
	
	If $\psi =[P,[\psi ]]\ $ and $\psi ^{\prime }=[P,[\psi ^{\prime }]]$ are sections of $\sum\nolimits^{\mathbb{C}}Q$ we will have:
	\begin{eqnarray*}
		\nabla _{X}^{\Sigma ^{\mathbb{C}}Q}\psi &=&\left[ P,X([\psi ])+\frac{1}{2}%
		\sum\nolimits_{i<j}w_{ij}(X)e_{i}e_{j}\cdot \lbrack \psi ]+\frac{1}{2}%
		iA^{l}(X)[\psi ]\right] , \\
		\left\langle \left\langle \nabla _{X}^{\Sigma ^{\mathbb{C}}Q}\psi ,\psi
		^{\prime }\right\rangle \right\rangle &=&\overline{[\psi ^{\prime }]}\left(
		X([\psi ])+\frac{1}{2}\sum\nolimits_{i<j}w_{ij}e_{i}e_{j}\cdot \lbrack \psi
		]+\frac{1}{2}iA^{l}(X)[\psi ]\right) , \\
		\left\langle \left\langle \psi ,\nabla _{X}^{\Sigma ^{\mathbb{C}}Q}\psi
		^{\prime }\right\rangle \right\rangle &=&\overline{\left( X([\psi ^{\prime
			}])+\frac{1}{2}\sum\nolimits_{i<j}w_{ij}e_{i}e_{j}[\psi ^{\prime }]+\frac{1%
			}{2}A^{l}[\psi ^{\prime }]\right) }[\psi ] \notag \\
		&=&\left( X(\overline{[\psi ^{\prime }]})+\frac{1}{2}\sum%
		\nolimits_{i<j}w_{ij}\overline{e_{i}e_{j}[\psi ^{\prime }]}+\frac{1}{2}%
		\overline{A^{l}}\overline{[\psi ^{\prime }]}\right) [\psi ] \notag \\
		&=&\left( X(\overline{[\psi ^{\prime }]})-\frac{1}{2}\sum%
		\nolimits_{i<j}w_{ij}\overline{[\psi ^{\prime }]}e_{i}e_{j}-\frac{1}{2}A^{l}%
		\overline{[\psi ^{\prime }]}\right) [\psi ],
	\end{eqnarray*}
	then
	\begin{eqnarray*}
		\left\langle \left\langle \nabla _{X}^{\Sigma ^{\mathbb{C}}Q}\psi ,\psi
		^{\prime }\right\rangle \right\rangle +\left\langle \left\langle \psi
		,\nabla _{X}^{\Sigma ^{\mathbb{C}}Q}\psi ^{\prime }\right\rangle
		\right\rangle &=&\overline{[\psi ^{\prime }]}X(\xi )+X(\overline{[\psi
			^{\prime }]})[\psi ], \\
		X\left\langle \left\langle \psi ,\psi ^{\prime }\right\rangle \right\rangle
		&=&X\left( \overline{\xi ^{\prime }}\xi \right) =X(\overline{\xi ^{\prime }}%
		)\xi +\overline{\xi ^{\prime }}X(\xi ).
	\end{eqnarray*}
\end{proof}

\begin{lemma}
The map $\left\langle \left\langle \cdot ,\cdot \right\rangle
	\right\rangle :\sum\nolimits^{\mathbb{C}}Q\times \sum\nolimits^{\mathbb{C}%
	}Q\rightarrow \mathbb{C}l_{(n+m)}$ satisfies:
	
	\begin{enumerate}
		\item $\left\langle \left\langle X\cdot \psi ,\varphi \right\rangle
		\right\rangle =-\left\langle \left\langle \psi ,X\cdot \varphi \right\rangle
		\right\rangle ,~\psi ,\varphi \in \sum\nolimits^{\mathbb{C}}Q,~X\in TQ.$
		
		\item $\tau \left\langle \left\langle \psi ,\varphi \right\rangle
		\right\rangle =\left\langle \left\langle \varphi ,\psi \right\rangle
		\right\rangle ,~\psi ,\varphi \in \sum\nolimits^{\mathbb{C}}Q$
	\end{enumerate}
\end{lemma}

	\begin{proof}
		\begin{enumerate}
			\item $\left\langle \left\langle X\cdot \psi ,\varphi \right\rangle
			\right\rangle =\tau \lbrack \varphi ][X\cdot \psi ]=\tau \lbrack \varphi
			][X][\psi ]=-\tau \lbrack \varphi ]\tau \lbrack X][\psi ]=-\left\langle
			\left\langle \psi ,X\cdot \varphi \right\rangle \right\rangle $
			
			\item $\tau \left\langle \left\langle \psi ,\varphi \right\rangle
			\right\rangle =\tau (\tau \lbrack \varphi ][\psi ])=\tau \lbrack \psi
			][\varphi ]=\left\langle \left\langle \varphi ,\psi \right\rangle
			\right\rangle .$
		\end{enumerate}
	\end{proof}

\begin{remark}
Note that the same is valid for the bundles $\left. \sum^{\mathbb{C}}Q \right|_M,$ $\sum^{\mathbb{C}}M$, $\sum^{\mathbb{C}}E$, $%
\sum\nolimits^{\mathbb{C}}.$
\end{remark}

\section{Spinorial representation of $Spin^{\mathbb{C}}$ submanifolds in $%
	\mathbb{R}^{n}$ by irreducible complex Clifford algebra spinors} \label{principalsection}

Consider here a immersion $M\hookrightarrow Q=\mathbb{R}^{n}.$ Since $\mathbb{R}^{n}$ is contratible, there is a global section $s:\mathbb{R}^{n}\rightarrow P_{Spin^{\mathbb{C}}_n}$, and the corresponding orthonormal basis $h=(E_{1},\cdots ,E_{n}):\mathbb{R}^{n}\rightarrow
P_{SO(n)},$ and $l^{\prime }:\mathbb{R}^{n}\rightarrow P_{S^{1}}$, where $(h,l^{\prime })=\Lambda ^{\mathbb{CR}^{n}}(s)\in \Gamma (P_{SO(n)}\times
P_{S^{1}}).$ In an adapted local section $\tilde{s}:U\subset M\subset 
\mathbb{R}^{n}\rightarrow \left. P_{Spin^{\mathbb{C}}_n}\right\vert
_{M}\subset P_{Spin^{\mathbb{C}}_n}$ we will denote the corresponding orthonormal local bases by $\tilde{h}=(e_{1},\cdots ,e_{n}):U\subset M\subset 
\mathbb{R}^{n}\rightarrow \left. P_{SO(n)}\right\vert _{M},$ and $l=\left.
l^{\prime }\right\vert _{M}:U\subset M\subset \mathbb{R}^{n}\rightarrow
\left. P_{S^{1}}\right\vert _{M}.$ Let $B:TM \times TM \rightarrow E$ the second fundamental form of that immersion.\\

\textbf{Case $n=2k$ even:}

\begin{lemma} \label{spinoresidealcasopar}
	Given an immersion $M\hookrightarrow Q=\mathbb{R}^{n},$ if $n=2k$ is even, we have $2^{k}$ classical spinors $\varphi _{i}\in $ $\Sigma^{\mathbb{C}}=\left. P_{Spin^{\mathbb{C}}_n}\right\vert _{M}\times _{\rho ^{n}}I^{n}$ \emph{(}comming from the restriction of an irreducible representation $\rho^{n}:\mathbb{C}l_{n}\rightarrow \End_{\mathbb{C}}\left(I^{n}\right)$\emph{)},  orthonormal, according $\left\langle \left\langle \cdot ,\cdot \right\rangle
	\right\rangle $, which satisfies the following equation:
	\begin{equation*}
	\nabla _{X}^{\Sigma^{\mathbb{C}}}\varphi _{i}=-\frac{1}{2}%
	\sum_{j=1}^{p}e_{j}\cdot B(X,e_{j})\cdot \varphi _{i}+\frac{1}{2}%
	\textbf{i}~A^{l}(X)\cdot \varphi _{i},~~i=1,\cdots ,2^{k}.
	\end{equation*}
\end{lemma}

\begin{proof}
 Fix the constant elements
	\begin{eqnarray*}
	\lbrack \varphi _{i}] &\in &I_{1}^{n}=\mathbb{C}l_{n}\mathrm{f}_{1}\subset \mathbb{C}l_{n},\text{ }i=1,\cdots ,2^{k},  \notag \\
	\text{such that }i_{2k}\left( [\varphi _{i}]\right)  &=& 
	\left( 
	\begin{array}{ccccc}
	0 & \cdots  & 0 & \cdots  & 0 \\ 
	\vdots  & \ddots  & \vdots  &  & \vdots  \\ 
	1 & \cdots  & 0 & \cdots  & 0 \\ 
	\vdots  &  & \vdots  & \ddots  & \vdots  \\ 
	0 & \cdots  & 0 & \cdots  & 0%
	\end{array}%
	\right) i\text{-th line},
	\end{eqnarray*}
and define the following spinorial fields $$\varphi _{i}=[s,[\varphi _{i}]]\in
	\sum\nolimits_{1}^{\mathbb{C}}\mathbb{R}^{n}:=P_{Spin^{\mathbb{C}}_n}\times _{\rho _{1}^{n}}I_{1}^{n}.$$ Remember that $w^{Q}(dh(X))=(w_{ij}^{h}(X))\in so(n),$ $\textbf{i}A(dl^{\prime }(X))=\textbf{i}A^{l^{\prime
	}}(X)\in \textbf{i}\mathbb{R},$ then
	\begin{eqnarray*}
	\nabla _{X}^{\Sigma_1 ^{\mathbb{C}}Q}\varphi _{i} &=&\left[ s,X([\varphi
	_{i}])+\left\{ \frac{1}{2}\sum\nolimits_{i<j}w_{ij}^{h}(X)E_{i}E_{j}+\frac{1%
	}{2}\textbf{i}~A^{l^{\prime }}(X)\right\} \cdot \lbrack \varphi _{i}]\right]   \notag
	\\
	&=&\left[ s,\frac{1}{2}\textbf{i}~A^{l^{\prime }}(X)\cdot \lbrack \varphi _{i}]\right]
	.~i=1,\cdots ,2^{k}.  \label{parallel3}
	\end{eqnarray*}
	
	In an adapted local section $$
	\tilde{s}:U\subset M\subset \mathbb{R}^{n}\rightarrow \left. P_{Spin^{\mathbb{C}}_n}\right\vert _{M}\subset
	P_{Spin^{\mathbb{C}}_n} , ~~
	\tilde{s}=s\cdot (g\otimes 1),g\in
	Spin_{n},1\in S^{1},$$
	with the corresponding orthonormal local bases
    \begin{equation*}
        \tilde{h}=(e_{1},\cdots ,e_{n}):U\subset M\subset \mathbb{R}^{n}\rightarrow \left. P_{SO(n)}\right\vert _{M}, \ \ \ l=\left.l^{\prime}\right\vert _{M}:U\subset M\subset \mathbb{R}^{n}\rightarrow \left.P_{S^{1}}\right\vert _{M},
	\end{equation*}
	eq.~\eqref{parallel3} can be written as
	\begin{eqnarray*}
	    \nabla _{X}^{\Sigma_1 ^{\mathbb{C}}Q}\varphi _{i} &=&\left[ \tilde{s},X(\widetilde{[\varphi ]})+\left\{ \frac{1}{2}\sum\nolimits_{i<j}w_{ij}^{%
		\tilde{h}}(X)e_{i}e_{j}+\frac{1}{2}\textbf{i}~A^{l}(X)\right\} \cdot \widetilde{%
		[\varphi ]}\right]   \notag \\&=&\left[\tilde{s},\frac{1}{2}\textbf{i}~A^{l}(X)\cdot (g\otimes1)^{-1}[\varphi_{i}]\right] =\left[\tilde{s},\frac{1}{2}\textbf{i}~A^{l}(X)\cdot \widetilde{[\varphi_{i}]}\right]   \notag \\
		&=&\frac{1}{2}\textbf{i}~A^{l}(X)\cdot \varphi _{i}.~i=1,\cdots ,2^{k}.
	\end{eqnarray*}
	Finally applying the spinorial Gauss formula eq. \eqref{gauss2}
	\begin{eqnarray*}
	\nabla _{X}^{\Sigma _{1}^{\mathbb{C}}Q}\varphi _{i}-\nabla _{X}^{\Sigma
		_{1}^{\mathbb{C}}}\varphi _{i} &=&\frac{1}{2}\sum_{j=1}^{p}e_{j}\cdot
	B(X,e_{j})\cdot \varphi _{i}  \notag \\
	\frac{1}{2}\textbf{i}~A^{l}(X)\cdot \varphi _{i}-\nabla _{X}^{\Sigma _{1}^{\mathbb{C%
	}}}\varphi _{i} &=&\frac{1}{2}\sum_{j=1}^{p}e_{j}\cdot B(X,e_{j})\cdot
	\varphi _{i}  \notag \\
	\nabla _{X}^{\Sigma _{1}^{\mathbb{C}}}\varphi _{i} &=&-\frac{1}{2}%
	\sum_{j=1}^{p}e_{j}\cdot B(X,e_{j})\cdot \varphi _{i}+\frac{1}{2}%
	\textbf{i}~A^{l}(X)\cdot \varphi _{i}.  
	\end{eqnarray*}
	
	Note that the spinors $\varphi _{i}=[\tilde{s},\widetilde{[\varphi _{i}]}]\in 
	$ $\sum _{1}^{\mathbb{C}}\subset \sum\nolimits_{1}^{\mathbb{C}}\mathbb{R}^{n},$ $i=1,\cdots ,2^{k}$, are orthonormal according to the product $\left\langle \left\langle \cdot ,\cdot \right\rangle \right\rangle $%
	\begin{eqnarray*}
	\left\langle \left\langle \varphi _{i},\varphi _{j}\right\rangle
	\right\rangle  &=&\tau \left( \widetilde{[\varphi _{j}]}\right) \widetilde{%
		[\varphi _{i}]}=\tau \left( (g\otimes 1)^{-1}[\varphi _{j}]\right) (g\otimes
	1)^{-1}[\varphi _{i}]  \notag \\
	&=&\tau \left( \lbrack \varphi _{j}]\right) \tau \left( (g\otimes
	1)^{-1}\right) (g\otimes 1)^{-1}[\varphi _{i}]  \notag \\
	&=&\tau \left( \lbrack \varphi _{j}]\right) [\varphi _{i}].  \notag \\
	\left\langle \left\langle \varphi _{i},\varphi _{i}\right\rangle
	\right\rangle  &=&1;~~\left\langle \left\langle \varphi _{i},\varphi
	_{j}\right\rangle \right\rangle =0,i\neq j, ~i,j =1,\cdots ,2^{k} 
	\end{eqnarray*}
\end{proof}

\begin{remark}
Considering the canonical isomorphisms $\mathbb{C}^{2^{k}}\simeq I_{1}^{n}\simeq \cdots \simeq
	I_{2^{k}}^{n},$ for a fixed $\widetilde{s},$ we can see each $[\varphi
	_{i}]\in I_{i}^{n}$ and $\widetilde{[\varphi _{i}]}:=\rho
	_{i}^{n}(g\otimes 1)^{-1}[\varphi _{i}]\in I_{i}^{n},~i=1,\cdots
	,2^{k}$ and note that $\tilde{[\varphi ]}=\widetilde{[\varphi
		_{1}]}+\cdots +\widetilde{[\varphi _{2^{k}}]}\in Spin^{\mathbb{C}}_n.$
\end{remark}

\textbf{Case }$n=2k+1$ \textbf{odd:}

\begin{lemma} \label{spinoresidealcasoimpar}
	Given an immersion $M\hookrightarrow Q=\mathbb{R}^{n},$ if $%
	n=2k+1$ is odd, we will have $2^{k}$ classical spinors $\varphi _{i;0}\in $ $\Sigma _{1;0}^{\mathbb{C}}=\left. P_{Spin^{\mathbb{C}%
		}_n}\right\vert _{M}\times _{\rho _{1;0}^{n}}I_{1;0}^{n}$ \emph{(}comming from the restriction of an irreducible representation $\rho_{1;0}^{n}:\mathbb{C}l_{n}\rightarrow \End_{\mathbb{C}}\left(I_{1;0}^{n}\right)$%
	\emph{)}, 
orthonormal, according $\left\langle \left\langle \cdot ,\cdot \right\rangle
	\right\rangle $, which satisfies the following equation:
	\begin{equation*}
	\nabla _{X}^{\Sigma _{1;0}^{\mathbb{C}}}\varphi _{i;0}=-\frac{1}{2}%
	\sum_{j=1}^{p}e_{j}\cdot B(X,e_{j})\cdot \varphi _{i;0}+\frac{1}{2}%
	\textbf{i}~A^{l}(X)\cdot \varphi _{i;0},~i=1,\cdots ,2^{k}.
	\end{equation*}
And we will also have $ 2^k $ classical spinors $\varphi _{i;1}\in $ $\Sigma _{1;1}^{\mathbb{C}}=\left. P_{Spin^{\mathbb{C}}_n}\right\vert _{M}\times _{\rho _{1;1}^{n}}I_{1;1}^{n}$ \emph{(}comming from the restriction of an irreducible representation $\rho
	_{1;1}^{n}:\mathbb{C}l_{n}\rightarrow \End_{\mathbb{C}}\left(I_{1;1}^{n}
	\right)$\emph{)}, 
	orthonormal, according $\left\langle \left\langle \cdot ,\cdot \right\rangle
	\right\rangle $,which satisfies the following equation:
	\begin{equation*}
	\nabla _{X}^{\Sigma _{1;1}^{\mathbb{C}}}\varphi _{i;1}=-\frac{1}{2}%
	\sum_{j=1}^{p}e_{j}\cdot B(X,e_{j})\cdot \varphi _{i;1}+\frac{1}{2}%
	\textbf{i}~A^{l}(X)\cdot \varphi _{i;1},i=1,\cdots ,2^{k}.
	\end{equation*}
\end{lemma}

\begin{proof}
The case $n=2k+1$ odd is completely analogous to the even. Fix the constant elements
	\begin{eqnarray*}
		\lbrack \varphi _{i;0}] &\in &I_{1;0}^{n}=\mathbb{C}l_{n}\mathrm{f}_{1;0}\subset \mathbb{C}l_{n},\text{ }i=1,\cdots ,2^{k}, \\
		i_{2k+1}\left( [\varphi _{i;0}]\right)  &=&
		\left( 
		\begin{array}{ccccc}
			0 & \cdots  & 0 & \cdots  & 0 \\ 
			\vdots  & \ddots  & \vdots  &  & \vdots  \\ 
			1 & \cdots  & 0 & \cdots  & 0 \\ 
			\vdots  &  & \vdots  & \ddots  & \vdots  \\ 
			0 & \cdots  & 0 & \cdots  & 0%
		\end{array}%
		\right) \oplus \left( 
		\begin{array}{ccccc}
			0 & \cdots  & 0 & \cdots  & 0 \\ 
			\vdots  & \ddots  & \vdots  &  & \vdots  \\ 
			0 & \cdots  & 0 & \cdots  & 0 \\ 
			\vdots  &  & \vdots  & \ddots  & \vdots  \\ 
			0 & \cdots  & 0 & \cdots  & 0%
		\end{array}%
		\right) i\text{-th line},
	\end{eqnarray*}%
	\begin{eqnarray*}
		\lbrack \varphi _{i;1}] &\in &I_{1;1}^{n}=\mathbb{C}l_{n}\mathrm{f}_{1;1}\subset \mathbb{C}l_{n},\text{ }i=1,\cdots ,2^{k}, \\
		i_{2k+1}\left( [\varphi _{i;1}]\right)  &=&%
		\left( 
		\begin{array}{ccccc}
			0 & \cdots  & 0 & \cdots  & 0 \\ 
			\vdots  & \ddots  & \vdots  &  & \vdots  \\ 
			0 & \cdots  & 0 & \cdots  & 0 \\ 
			\vdots  &  & \vdots  & \ddots  & \vdots  \\ 
			0 & \cdots  & 0 & \cdots  & 0%
		\end{array}%
		\right) \oplus \left( 
		\begin{array}{ccccc}
			0 & \cdots  & 0 & \cdots  & 0 \\ 
			\vdots  & \ddots  & \vdots  &  & \vdots  \\ 
			1 & \cdots  & 0 & \cdots  & 0 \\ 
			\vdots  &  & \vdots  & \ddots  & \vdots  \\ 
			0 & \cdots  & 0 & \cdots  & 0%
		\end{array}%
		\right) i\text{-th line},
	\end{eqnarray*}%
	and define the following spinorial fields 
	\begin{align*}
	  & \varphi _{i;0}=[s,[\varphi _{i;0}]]\in
	\sum\nolimits_{1;0}^{\mathbb{C}}\mathbb{R}^{n}:=P_{Spin^{\mathbb{C}}_n}\times _{\rho _{1;0}^{n}}I_{1;0}^{n}, \\ &\varphi _{i;1}=[s,[\varphi _{i;1}]]\in
	\sum\nolimits_{1;1}^{\mathbb{C}}\mathbb{R}^{n}:=P_{Spin^{\mathbb{C}%
		}_n}\times _{\rho _{1;1}^{n}}I_{1;1}^{n},
	\end{align*} which in an adapted local section are written as $\varphi _{i;0}=[\tilde{s},\widetilde{[\varphi
		_{i;0}]}],$ $\varphi _{i;1}=[\tilde{s},\widetilde{[\varphi
		_{i;1}]}]$ and satisfy
	\begin{eqnarray*} \label{killingimpar1}
	\nabla _{X}^{\Sigma _{1;0}^{\mathbb{C}}}\varphi _{i;0}=-\frac{1}{2}%
	\sum_{j=1}^{p}e_{j}\cdot B(X,e_{j})\cdot \varphi _{i;0}+\frac{1}{2}%
	\textbf{i}~A^{l}(X)\cdot \varphi _{i;0},~i=1,\cdots ,2^{k}, \notag \\
	\nabla _{X}^{\Sigma _{1;1}^{\mathbb{C}}}\varphi _{i;1}=-\frac{1}{2}%
	\sum_{j=1}^{p}e_{j}\cdot B(X,e_{j})\cdot \varphi _{i;1}+\frac{1}{2}%
	\textbf{i}~A^{l}(X)\cdot \varphi _{i;1},~i=1,\cdots ,2^{k}.
	\end{eqnarray*}
	
	Note that the spinors $\varphi _{i;\lambda}=[\tilde{s},\widetilde{[\varphi _{i;\lambda}]}]\in 
	$ $\sum\nolimits_{1;\lambda}^{\mathbb{C}}\subset \sum\nolimits_{1;\lambda}^{\mathbb{C}}\mathbb{R}^{n},$ $i=1,\cdots ,2^{k} ; \lambda = 0,1$, are also orthonormal according to the product $\left\langle \left\langle \cdot ,\cdot \right\rangle
	\right\rangle .$ 
\end{proof}

\begin{remark}
Considering the canonic isomorphisms $\mathbb{C}^{2^{k}}\simeq I_{1;0}\simeq \cdots \simeq
	I_{2^{k};0}\simeq I_{1;1}\simeq \cdots \simeq I_{2^{k};1},$ fixed $\tilde{s}%
	,$ we can see each $[\varphi _{i;\lambda}]\in I_{i;\lambda}$ and $\widetilde{[\varphi _{i;\lambda}]%
	}:=\rho _{i;\lambda}(g\otimes 1)^{-1}[\varphi _{i;\lambda}]\in I_{i;\lambda},$ $i=1,\cdots ,2^{k}; \lambda=0,1$ and note that $\widetilde{[\varphi ]}=\widetilde{[\varphi
		_{1;0}]}+\cdots +\widetilde{[\varphi _{2^{k};0}]}+\widetilde{[\varphi _{1;1}]}%
	+\cdots +\widetilde{[\varphi _{2^{k};1}]}\in Spin^{\mathbb{C}}_n.$
\end{remark}

\bigskip 

\textbf{Then there is the reciprocal question: Given this set of orthonormal spinors is it possible to construct an isometric immersion of the manifold} $M$ \textbf{in} $\mathbb{R}^{n}$ \textbf{?}

\bigskip

Let $M$ a riemannian $p$-dimensional manifold, $E\rightarrow M$ a vector bundle over $\mathbb{R}$ with rank $q$, assume that $TM$ and $E$ are oriented and $Spin^{\mathbb{C}}.$ Denote again by $P_{SO(p)}$
the frame bundle of $TM$ and by $P_{SO(q)}$ the frame bundle of $E.$
Also the respective $Spin^{\mathbb{C}}$  structures are represented as
\begin{equation*}
	\Lambda ^{1\mathbb{C}} :P_{Spin_{p}^{\mathbb{C}}}\rightarrow
	P_{SO(p)}\times P_{S^{1}}^{1}, \ \ \ \
	\Lambda ^{2\mathbb{C}} :P_{Spin_{q}^{\mathbb{C}}}\rightarrow
	P_{SO(q)}\times P_{S^{1}}^{2}.
\end{equation*}%
Define here, as well as in the section
 \ref{512}, the $S^{1}$-principal bundle $P_{S^{1}}$ such as the bundle whose transition functions are defined as the product of the transition functions of $P_{S^{1}}^{1}$ and $P_{S^{1}}^{2}$. It is not difficult to see how that there is a canonical morphism between bundles: $\Phi
:P_{S^{1}}^{1}\times _{M}P_{S^{1}}^{2}\rightarrow P_{S^{1}}$ with $\Phi
(p^{1}\cdot s^{1},p^{2}\cdot s^{2})=\Phi (p^{1},p^{2})s^{1}s^{2},$ $p_{1}\in
P_{S^{1}}^{1},$ $p_{2}\in P_{S^{1}}^{2},$ $s_{1},s_{2}\in S^{1}.$

Here $\textbf{i}A^{1}:TP_{S^{1}}^{1}\rightarrow \textbf{i}\mathbb{R}$, $\textbf{i}A^{2}:TP_{S^{1}}^{2}%
\rightarrow \textbf{i}\mathbb{R}$ are arbitrary connections in $P_{S^{1}}^{1}$ and $P_{S^{1}}^{2}$. Fix the following local sections $s=(e_{1},\cdots ,e_{p}):U\rightarrow P_{SO(p)}$, $l_{1}:U\rightarrow
P_{S^{1}}^{1}$, $l_{2}:U\rightarrow P_{S^{1}}^{2},$ $l=\Phi
(l_{1},l_{2}):U\rightarrow P_{S^{1}}$. Now $\textbf{i}A:TP_{S^{1}}\rightarrow \textbf{i}%
\mathbb{R}$ is the connection defined by $\textbf{i}A(d\Phi
(l_{1},l_{2}))=\textbf{i}A_{1}(dl_{1})+\textbf{i}A_{2}(dl_{2}).$

Here we will fix the following complexed spinor bundles
\begin{eqnarray*}
	\sum\nolimits_{i}^{\mathbb{C}} &:=&\left( P_{Spin^{\mathbb{C}}_p}\times
	_{M}P_{Spin^{\mathbb{C}}_q}\right) \times _{\rho _{i}^{n}}I_{i}^{n},~%
	\text{if }n=2k\text{ is even}\\
	\sum\nolimits_{i;\lambda}^{\mathbb{C}} &:=&\left( P_{Spin^{\mathbb{C}}_p}\times
	_{M}P_{Spin^{\mathbb{C}}_q}\right) \times _{\rho _{i;\lambda}^{n}}I_{i,\lambda}^{n},~%
	\text{if }n=2k+1\text{ is odd} \\
	i &=&1,\cdots ,2^{k};\lambda=0,1.
\end{eqnarray*}

For the \textbf{case }$n=2k$\textbf{\ even}: suppose that there are $2^{k} $ orthonormal spinors $\varphi _{i}\in \Gamma \left( \Sigma _{1}^{\mathbb{C}}\right) $ satisfying the following equation
\begin{equation*}
\nabla _{X}^{\Sigma _{1}^{\mathbb{C}}}\varphi _{i}=-\frac{1}{2}%
\sum_{j=1}^{p}e_{j}\cdot B(X,e_{j})\cdot \varphi _{i}+\frac{1}{2}%
\textbf{i}~A^{l}(X)\cdot \varphi _{i},i=1,\cdots ,2^{k},  \label{killingi}
\end{equation*}%
where $B:TM\times TM\rightarrow E$ is a symmetric and bilinear form.

\begin{remark}
	Given the natural isomorphisms $\Sigma _{1}^{\mathbb{C}}\simeq
	\cdots \simeq \Sigma _{2^{k}}^{\mathbb{C}}$, we can consider each $\varphi _{i}=\left[\tilde{s},[\varphi _{i}]\right] \in \Gamma \left( \Sigma _{i}^{\mathbb{C}}\right), $ which will be solutions of 
\begin{equation*}
\nabla _{X}^{\Sigma _{i}^{\mathbb{C}}}\varphi _{i}=-\frac{1}{2}\sum_{j=1}^{p}e_{j}\cdot B(X,e_{j})\cdot \varphi _{i}+\frac{1}{2}%
\textbf{i}~A^{l}(X)\cdot \varphi _{i},i=1,\cdots ,2^{k}. 
\end{equation*}%

Without loss of generality, since we can make an adequate linear combination of eq.~(\ref{killingi}), fixed $\tilde{s}\in \Gamma \left( P_{Spin^{\mathbb{C}}(p)}\times _{M}P_{Spin^{\mathbb{C}}(q)}\right) $ we have 
\begin{equation}
\lbrack \varphi _{1}]+\cdots +[\varphi _{2^{k}}]\in Spin_{n}^{\mathbb{C}%
}\subset I_{1}^{n}\oplus \cdots \oplus I_{2^{k}}^{n}=\mathbb{C}l_{n}.
\label{somaestaemspin}
\end{equation}

Note that this does not depend on the choice of referential $ \tilde{s} $, since we are working with $Spin^{\mathbb{C}}$-principal bundles . Thus in another spinorial frame eq.~\eqref{somaestaemspin} remains valid.

\end{remark}

Then we can define the following $\mathbb{C}$-valued $1$-forms :
\begin{eqnarray}
\xi _{ij} &:&TM\rightarrow \mathrm{f}_{i}\mathbb{C}l_{n}\mathrm{f}%
_{j}\simeq \mathbb{C} \notag \\
\xi _{ij}(X) &=&\left\langle \left\langle X\cdot \varphi _{i},\varphi
_{j}\right\rangle \right\rangle ,i,j=1,\cdots ,2^{k}.
\end{eqnarray}

Now, since we are assuming that the $\varphi _{i}$ are such that
the equations \eqref{killingi} and \eqref{somaestaemspin} are valid, we define the following $\mathbb{C}l_{n}$-valued $1$-form 
\begin{equation} \label{xisoma}
\xi (X) = \sum_{i,j=1}^{2^k} \xi_{ij} (X),  \ \  \ \
\xi (X) \in \bigoplus_{i,j=1}^{2^k} \mathrm{f}_{i}\mathbb{C}l_{n}\mathrm{f}_{j} = \mathbb{C}l_{n}.
\end{equation}

\begin{lemma}
	\label{lemaxii}Suppose each $\varphi _{i}\in \Gamma \left( \Sigma
	_{1}^{\mathbb{C}}\right) \simeq \Gamma \left( \Sigma _{i}^{\mathbb{C}%
	}\right) $ satisfy eqs.~\eqref{killingi} and \eqref{somaestaemspin}, then $\xi $ defined by eq.~(\ref{xisoma}) is such that
	
	\begin{enumerate}
		\item $\xi $ is a $\mathbb{R}^{n}$-valued $1$-form. 
		
		\item $\xi $ is a closed $1$-form, $d\xi =0.$
	\end{enumerate}
\end{lemma}
		
		\begin{proof}
			\begin{enumerate}
				\item If $\varphi =[\tilde{s},[\varphi_i ]],X=[\tilde{s},[X]],$ where $[\varphi_i ]$ and $[X]$
				represent $\varphi_i $ in a given frame $\tilde{s}\in \Gamma
				\left( P_{Spin^{\mathbb{C}}_p}\times P_{Spin^{\mathbb{C}}_q}\right),~i=1,\cdots , 2^{k},$%
				\begin{eqnarray*}
					 \xi (X) && :=\sum_{i,j=1}^{2^k} \xi_{ij} (X)= \sum_{i,j=1}^{2^k} \tau \lbrack \varphi_j ][X][\varphi_i ] = \left( \sum_{j=1}^{2^k}  \tau \lbrack \varphi_j ] \right) [X] \left( \sum_{i=1}^{2^k} [\varphi_i ] \right) \\ && = \tau \lbrack \varphi ][X][\varphi ]\in \mathbb{R}^{n}\subset
				Cl_{n}\subset \mathbb{C}l_{n}, \text{ since } [\varphi ]\in Spin^{\mathbb{C}}.
				\end{eqnarray*}
				
				\item For simplicity suppose that at the arbitrary point $x_{0}\in M$ have $\nabla
				^{M}X=\nabla ^{M}Y=0,$ and write $\nabla _{X}^{\Sigma ^{\mathbb{C}%
				}}\varphi =\nabla _{X}\varphi $ and $\nabla ^{M}X=\nabla X$,%
				\begin{eqnarray*}
					X(\xi (Y)) &=& X\bigg(\sum_{i,j=1}^{2^k} \xi_{ij} (Y)\bigg)  =\sum_{i,j=1}^{2^k} \bigg( \left\langle \left\langle Y\cdot \nabla _{X}\varphi_i ,\varphi_j
					\right\rangle \right\rangle +\left\langle \left\langle Y\cdot \varphi_i
					,\nabla _{X}\varphi_j \right\rangle \right\rangle \bigg) \\ &=&\sum_{i,j=1}^{2^k}(id-\tau )\left\langle
					\left\langle Y\cdot \varphi_i ,\nabla _{X}\varphi_j \right\rangle \right\rangle
					 \\
					 &=&\sum_{i,j=1}^{2^k}(id-\tau )\left\langle \left\langle \varphi_i ,\frac{1}{2}
					\sum_{r=1}^{p}Y\cdot e_{r}\cdot B(X,e_{r})\cdot \varphi_j  -\frac{1}{2}\mathbf{i}~A^{l}(X) Y\cdot \varphi_j \right\rangle \right\rangle ,
					\\
					Y(\xi (X)) &=&\sum_{i,j=1}^{2^k}(id-\tau )\left\langle \left\langle \varphi_i ,\frac{1}{2}
					\sum_{r=1}^{p}X\cdot e_{r}\cdot B(Y,e_{r})\cdot \varphi_j  -\frac{1}{2}\mathbf{i}~A^{l}(Y)X\cdot \varphi_j \right\rangle \right\rangle . 
				\end{eqnarray*}
				From here it follows that
				\begin{eqnarray}
					d\xi (X,Y) &=& X(\xi (Y))-Y(\xi (X)) \notag \\
					&=&\sum_{i,j=1}^{2^k}(id-\tau )\left\langle \left\langle \varphi_i ,\frac{1}{2}%
					\sum_{r=1}^{p}\left[ Y\cdot e_{r}\cdot B(X,e_{r})\right.\right. \right. \notag \\ && \left. \left. \left. -X\cdot
					e_{r}\cdot B(Y,e_{r})\right] \cdot \varphi_j  +\frac{1}{2}\mathbf{i}\left(
					A^{l}(Y)X-A^{l}(X)Y\right) \cdot \varphi_j \right\rangle \right\rangle \notag \\
					&=&\sum_{i,j=1}^{2^k}(id-\tau )\left\langle \left\langle \varphi _i,C\cdot \varphi_j \right\rangle
					\right\rangle , \notag
				\end{eqnarray}
				$$C:=\frac{1}{2}\sum_{r=1}^{p}\left[ Y\cdot e_{r}\cdot B(X,e_{r})-X\cdot
				e_{r}\cdot B(Y,e_{r})\right] +\frac{1}{2}\mathbf{i} \big( A^{l}(Y)X-A^{l}(X)Y \big).$$
				Write $X=\sum_{s=1}^{p}x^{s}e_{s};~Y=\sum_{s=1}^{p}y^{s}e_{s}$ then
				\begin{eqnarray}
					\sum_{r=1}^{p}X\cdot e_{r}\cdot B(Y,e_{r})
					&=&\sum_{r=1}^{p}\sum_{s=1}^{p}x^{s}e_{s}\cdot
					e_{r}\cdot B(Y,e_{r}) \notag \\ &=& -B(Y,X)+\sum_{r=1}^{p}\sum
					_{\substack{ s=1  \\ s\neq r}}^{p}x^{s}e_{s}\cdot e_{r}\cdot B(Y,e_{r}), \notag \\
					\sum_{r=1}^{p}Y\cdot e_{r}\cdot B(X,e_{r})	&=&\sum_{r=1}^{p}\sum_{s=1}^{p}y^{s}e_{s}\cdot
					e_{r}\cdot B(X,e_{r}) \notag \\
					&=&-B(X,Y)+\sum_{r=1}^{p}\sum
					_{\substack{ s=1  \\ s\neq r}}^{p}y^{s}e_{s}\cdot e_{r}\cdot B(X,e_{r}), \notag
				\end{eqnarray}
				from which we conclude
				\begin{eqnarray}
					C &=&\frac{1}{2}\left[ \sum_{r=1}^{p}\sum_{\substack{ 
							s=1  \\ s\neq r}}^{p}e_{s}\cdot e_{r}\cdot \left[
					y^{s}B(X,e_{r})-x^{s}B(Y,e_{r})\right] \notag \right]  +\frac{\mathbf{i}}{2}\bigg( A^{l}(Y)X-A^{l}(X)Y \bigg)
					\end{eqnarray}
					\begin{eqnarray*}
					&&\tau ([C]) -\frac{1}{2}\left[ \sum_{r=1}^{p}\sum
					_{\substack{ s=1  \\ s\neq r}}^{p}\left[ y^{s}B(X,e_{r})-x^{s}B(Y,e_{r})%
					\right] \right] \cdot e_{r}\cdot e_{s} \notag  +\frac{\mathbf{i}}{2}\bigg( A^{l}(Y)\left[ X\right]
					-A^{l}(X)\left[ Y\right] \bigg) \notag \\
					&&=\frac{1}{2}\left[ \sum_{r=1}^{p}\sum_{\substack{ s=1 
							\\ s\neq r}}^{p}e_{s}\cdot e_{r}\cdot \left[ y^{s}B(X,e_{r})-x^{s}B(Y,e_{r})
					\right] \right]  + \frac{\mathbf{i}}{2}\bigg( A^{l}(Y)\left[ X\right] -A^{l}(X)\left[ Y\right]
					\bigg)=[C].
				\end{eqnarray*}
			What implies that
				\begin{eqnarray}
				d\xi (X,Y)&=&\sum_{i,j=1}^{2^k}(id-\tau )\left\langle \left\langle \varphi_i ,C\cdot \varphi_j
				\right\rangle \right\rangle =\sum_{i,j=1}^{2^k}(id-\tau )(\tau \lbrack \varphi_j ]\tau \lbrack
				C][\varphi_i ])\notag \\&=&(id-\tau )\bigg(\sum_{j=1}^{2^k}\tau \lbrack \varphi_j ]\tau \lbrack
				C]\sum_{i=1}^{2^k}[\varphi_i ]\bigg)
				\notag =(id-\tau )(\tau \lbrack \varphi ]\tau \lbrack
				C][\varphi ])=0. 
				\end{eqnarray}
			\end{enumerate}
		\end{proof}

\textbf{For the case }$n=2k+1$\textbf{ odd}: suppose now that there are $2^{k+1} $ spinors $\varphi _{i;0}\in \Gamma \left( \Sigma _{1;0}^{\mathbb{C}}\right) $, $\varphi _{i;1}\in \Gamma \left( \Sigma _{1;1}^{\mathbb{C}}\right) $ which satisfy the following equations
\begin{eqnarray} \label{killingi2}
\nabla _{X}^{\Sigma _{1;0}^{\mathbb{C}}}\varphi _{i;0}=-\frac{1}{2}\sum_{j=1}^{p}e_{j}\cdot B(X,e_{j})\cdot \varphi _{i;0}+\frac{1}{2}\textbf{i}~A^{l}(X)\cdot \varphi _{i;0},~i=1,\cdots ,2^{k},  \notag \\
\nabla _{X}^{\Sigma _{1;1}^{\mathbb{C}}}\varphi _{i;1}=-\frac{1}{2}\sum_{j=1}^{p}e_{j}\cdot B(X,e_{j})\cdot \varphi _{i;1}+\frac{1}{2}%
\textbf{i}~A^{l}(X)\cdot \varphi _{i;1},~i=1,\cdots ,2^{k},
\end{eqnarray}%
where $B:TM\times TM\rightarrow E$ is symmetric bilinear form.

\begin{remark}
Given the natural isomorphisms $\Sigma _{1;0}^{\mathbb{C}}\simeq\cdots \simeq \Sigma _{2^{k};0}^{\mathbb{C}}$ and $\Sigma _{1;1}^{\mathbb{C}}\simeq\cdots \simeq \Sigma _{2^{k};1}^{\mathbb{C}}$ we can consider each $\varphi _{i;\lambda}=\left[ \tilde{s},[\varphi _{i;\lambda}]	\right] \in \Gamma \left( \Sigma _{i;\lambda}^{\mathbb{C}}\right),$ which will be solutions of
\begin{equation*} 
\nabla _{X}^{\Sigma _{i;\lambda}^{\mathbb{C}}}\varphi _{i;\lambda}=-\frac{1}{2}\sum_{j=1}^{n}e_{j}\cdot B(X,e_{j})\cdot \varphi _{i;\lambda}+\frac{1}{2}\textbf{i}~A^{l}(X)\cdot \varphi _{i;\lambda},i=1,\cdots ,2^{k},\lambda=0,1.  
\end{equation*}

Without loss of generality, since we can make an adequate linear combination of eq.~(\ref{killingi2}), fixed $\tilde{s}\in \Gamma \left( P_{Spin^{\mathbb{C}}_p}\times _{M}P_{Spin^{\mathbb{C}}_q}\right) $ we have
\begin{align} 
&\lbrack \varphi _{1;0} \rbrack +\cdots +\lbrack \varphi _{2^{k};0} \rbrack + \lbrack \varphi _{1;1} \rbrack +\cdots +\lbrack \varphi _{2^{k};1} \rbrack \in Spin_{n}^{\mathbb{C}} \notag \\ &\subset \left( I_{1;1}\oplus \cdots \oplus I_{2^{k};1}\right) \oplus \left( I_{1;2}\oplus \cdots \oplus I_{2^{k};2} \right)=\mathbb{C}l_{n}.
\label{somaestaemspin2}
\end{align}

Note that this does not depend on the choice of referential $ \tilde{s} $, since we are working with $Spin^{\mathbb{C}}$-principal bundles . Thus in another spinorial frame eq.~\eqref{somaestaemspin2} remains valid.
\end{remark}

We can thus define the following $\mathbb{C}$-valued $1$-forms :
\begin{eqnarray*}
\xi _{ij;\lambda} &:&TM\rightarrow \mathrm{f}_{i;\lambda}\mathbb{C}l_{(n+m)}\mathrm{f}%
_{j;\lambda}\simeq \mathbb{C} \notag \\
\xi _{ij;\lambda}(X) &=&\left\langle \left\langle X\cdot \varphi _{i;\lambda},\varphi
_{j;\lambda}\right\rangle \right\rangle ,i,j=1,\cdots ,2^{k},\lambda=0,1.
\end{eqnarray*}

Now, since we are assuming that the $\varphi _{i}$ are such that
the equations \eqref{killingi2} and \eqref{somaestaemspin2} are valid, we define the following $\mathbb{C}l_{n}$-valued $1$-form 
\begin{eqnarray*} \label{xisoma2}
\xi (X) = \sum_{i,j=1}^{2^k} \sum_{\lambda=0} ^1 \xi_{ij;\lambda} (X), \ \ \  \
\xi (X) \in \bigoplus_{i,j=1}^{2^k} \bigoplus_{\lambda=0}^1 \mathrm{f}_{i;\lambda}\mathbb{C}l_{n}\mathrm{f}_{j;\lambda} = \mathbb{C}l_{n}.
\end{eqnarray*}

\begin{lemma}
	\label{lemaxii2}Suppose each $\varphi _{i;0}\in \Gamma \left( \Sigma
	_{1;0}^{\mathbb{C}}\right) \simeq \Gamma \left( \Sigma _{i;0}^{\mathbb{C}}\right) $ and $\varphi _{i;1}\in \Gamma \left( \Sigma
	_{1;1}^{\mathbb{C}}\right) \simeq \Gamma \left( \Sigma _{i;1}^{\mathbb{C}}\right) $ satisfy eqs.~\eqref{killingi2} and \eqref{somaestaemspin2}, then $\xi $ defined by eq.~(\ref{xisoma2}) is such that
	
	\begin{enumerate}
		\item $\xi $ is a $\mathbb{R}^{n}$-valued $1$-form. 
		
		\item $\xi $ is a closed $1$-form, $d\xi =0.$
	\end{enumerate}
\end{lemma}

\begin{proof}
The proof is exactly the same as the lemma \ref{lemaxii}. Equation \eqref{somaestaemspin2} implies that $\xi$ is a $\mathbb{R}^{n}$-valued form. To show that $\xi$ is a closed form we use that each $\varphi_{i;\lambda}$ satisfy eq.~\eqref{killingi2}.
\end{proof}

\bigskip

Thus, regardless of the parity of $ n $, if we assume that the $ M $ is simply connected, by the Poincaré`s lemma follows that there is a function
\begin{equation*}
F:M\rightarrow \mathbb{R}^{n},
\end{equation*}%
such that $dF=\xi .$ 

In addition the following lemma is valid

\begin{lemma} \label{lemaimersaoo} With the above considerations the following items are valid
	\begin{enumerate}
		\item The map $F:M\rightarrow \mathbb{R}%
		^{n+m}$ is an isometry.
		
		\item The map
		\begin{eqnarray*}
			\Phi _{E} &:&E\rightarrow M\times \mathbb{R}^{n} \\
			X &\in &E_{m}\mapsto (F(m),\xi (X))
		\end{eqnarray*}%
		is an isometry between $E$ and the normal bundle of $F(M)$ in $\mathbb{R}^{ n },$ preserving the connection and second fundamental form.
	\end{enumerate}
\end{lemma}

\begin{proof}
			\begin{enumerate}
				\item Let $X,Y\in \Gamma (TM\oplus E),$ consequently
				\begin{eqnarray*}
					&&\left\langle \xi (X),\xi (Y)\right\rangle =-\frac{1}{2}\big( \xi (X)\xi
					(Y)-\xi (Y)\xi (X)\big)  \\ &&=-\frac{1}{2}  \bigg( \sum_{i,j=1}^{2^k}\xi_{ij} (X) \sum_{\alpha,\beta=1}^{2^k} \xi_{\alpha\beta}(Y) -\sum_{\alpha,\beta=1}^{2^k}\xi_{\alpha\beta} (Y) \sum_{i,j=1}^{2^k} \xi_{ij} (X)\bigg) \notag \\
					&&=-\frac{1}{2}   \bigg( \sum_{i,j=1}^{2^k}\tau \lbrack \varphi_j
					][X][\varphi_i ]\sum_{\alpha,\beta=1}^{2^k}\tau \lbrack \varphi_\beta ][Y][\varphi_\alpha ]-\sum_{\alpha,\beta=1}^{2^k}\tau \lbrack \varphi_\beta
					][Y][\varphi_\alpha ] \sum_{i,j=1}^{2^k}\tau \lbrack \varphi_j ][X][\varphi_i ]\bigg)\\
					&&=-\frac{1}{2}   \bigg( \sum_{j=1}^{2^k}\tau \lbrack \varphi_j
					][X]\sum_{i=1}^{2^k}[\varphi_i ]\sum_{\beta=1}^{2^k}\tau \lbrack \varphi_\beta ][Y]\sum_{\alpha=1}^{2^k}[\varphi_\alpha ] -\sum_{\beta=1}^{2^k}\tau \lbrack \varphi_\beta	][Y]\sum_{\alpha=1}^{2^k}[\varphi_\alpha ] \sum_{j=1}^{2^k}\tau \lbrack \varphi_j ][X]\sum_{i=1}^{2^k}[\varphi_i ]\bigg) \\
					&& =  -\frac{1}{2}\tau \lbrack \varphi ]\Big( [X]\tau \lbrack \varphi ][\varphi][Y]-[Y]\tau \lbrack \varphi ][\varphi][X]\Big) [\varphi
					]  = \tau \lbrack \varphi ] \left\langle X,Y\right\rangle
					[\varphi ] \\ 
					&&= -\frac{1}{2}\tau \lbrack \varphi ]\Big( [X][Y]-[Y][X]\Big) [\varphi
					] \notag =\tau \lbrack \varphi ] \left\langle X,Y\right\rangle 
					[\varphi ] =\left\langle X,Y\right\rangle \tau \lbrack \varphi ][\varphi
					]=\left\langle X,Y\right\rangle .
				\end{eqnarray*}
				This implies that $F$ is an isometry with its image, and that $\Phi _{E}$ is a bundle map between $E$ and the normal bundle of $F(M)$ in $\mathbb{R}^{n}$ which preserves the metric.
				
				\item Denote by $B_{F}$ and $\nabla ^{\prime F}$ the second fundamental form and the normal connection of immersion $F$ respectively. We would like to show that:
				\begin{equation*}
					\mathbf{i)}\xi (B(X,Y)) =B_{F}(\xi (X),\xi (Y)), 
				\ \ \ \ \ \ \ \ 	\mathbf{ii)}\xi (\nabla _{X}^{\prime }\eta ) =\nabla _{\xi (X)}^{\prime F}\xi
					(\eta ),
				\end{equation*}
				for all $X,Y\in \Gamma (TM)$ and $\eta \in \Gamma (E)$.
				\\
				$\mathbf{i)}$ First note that:
				\begin{equation*}
				B^{F}(\xi (X),\xi (Y)):=\{\nabla _{\xi (X)}^{F}\xi (Y)\}^{\bot }=\{X(\xi
				(Y))\}^{\bot }, 
				\end{equation*}
 				where the symbol $\bot $ means that we are considering the vector component that is orthonormal to the immersion. We know that
				\begin{multline*}
				X(\xi (Y))=\sum_{i,j=1}^{2^k} (id-\tau )\left\langle \left\langle \varphi_i,\frac{1}{2}\sum_{r=1}^{p}Y\cdot e_{r}\cdot B(X,e_{r})\cdot \varphi_j -\frac{1}{2}\mathbf{i} ~A^{l}(X)Y\cdot \varphi_j \right\rangle \right\rangle  
				\\ \shoveleft{= \sum_{i,j=1}^{2^k} (id-\tau )\left\langle \left\langle \varphi_i ,\frac{1}{2} \left( \sum_{r=1}^{p}\sum_{s=1}^{s}y^{s}e_{s}\cdot e_{r}\cdot
				B(X,e_{r})- \mathbf{i}~A^{l}(X)Y \right) \cdot \varphi_j \right\rangle \right\rangle }  \\
				\shoveleft{=\sum_{i,j=1}^{2^k} (id-\tau )\left\langle \left\langle \varphi_i ,\frac{1}{2} \left(
				\sum_{r=1}^{p}y^{r}e_{r}\cdot e_{r}\cdot
				B(X,e_{r}) \right. \right. \right.  } \\
				\shoveright{ \left.\left. \left. +\sum_{r=1}^{p}\sum_{s=1,s\neq r}^{p}y^{s}e_{s}\cdot e_{r}\cdot B(X,e_{r})-\mathbf{i}~A^{l}(X)Y \right) \cdot \varphi_j
				\right\rangle \right\rangle } \\
				\shoveleft{=\sum_{i,j=1}^{2^k} (id-\tau )\left\langle \left\langle \varphi_i ,\frac{1}{2}\left(
				-B(X,Y)+D\right) \cdot \varphi_j \right\rangle \right\rangle .} \\
				\end{multline*}
				\begin{equation*}
					D=\sum_{r=1}^{p}\sum_{s=1,s\neq r}^{p}y^{s}e_{s}\cdot
					e_{r}\cdot B(X,e_{r})-\mathbf{i}~A^l (X)Y, \ \ \ \ 
					\tau \lbrack D] =[D].
				\end{equation*}
				Consequently
				\begin{eqnarray*}
					X(\xi (Y)) &=&\frac{1}{2}(id-\tau )\left\langle \left\langle \sum_{i=1}^{2^k} \varphi_i ,\left(
					-B(X,Y)+D\right) \cdot \sum_{j=1}^{2^k} \varphi_j \right\rangle \right\rangle \notag \\
					&=& \sum_{i=1}^{2^k} \tau \lbrack \varphi _j]\Big( -\tau \lbrack B(X,Y)\rbrack+\tau \lbrack D\rbrack \Big) \sum_{i=1}^{2^k}[\varphi_i ]\\&=&-\tau \lbrack \varphi ]\tau \lbrack B(X,Y)][\varphi ] =\left\langle
					\left\langle \varphi ,B(X,Y)\cdot \varphi \right\rangle \right\rangle =\xi (B(X,Y)).
				\end{eqnarray*}
				Therefore, we conclude that
				\begin{eqnarray*}
					B^{F}(\xi (X),\xi (Y)) &=&B^{F}(\xi (X),\xi (Y)):=\{\nabla _{\xi(X)}^{F}\xi(Y)\}^{\bot}\\
					&=&\{X(\xi (Y))\}^{\bot } 
					=\{\xi (B(X,Y))\}^{\bot }=\xi (B(X,Y)),
				\end{eqnarray*}
				here was used the fact that $F=\int \xi $ is an isometry: $B(X,Y)\in
				E\Rightarrow \xi (B(X,Y))\in TF(M)^{\bot }.$ Therefore, the statement $\mathbf{i)}$ follows.
				\\ 
				$\mathbf{ii)}$ Firstly note that
				\begin{eqnarray*}
				\nabla _{\xi (X)}^{F}\xi (\eta )&=&\left\{ X(\xi (\eta ))\right\} ^{\bot
				}=\left\{\sum_{i,j=1}^{2^k} X\left\langle \left\langle \eta \cdot \varphi_i ,\varphi_j
				\right\rangle \right\rangle \right\} ^{\bot } \notag \\
				&=& \sum_{i,j=1}^{2^k}\left\langle \left\langle \eta \cdot \nabla _{X}\varphi _i,\varphi_j
				\right\rangle \right\rangle ^{\bot }+\sum_{i,j=1}^{2^k}\left\langle \left\langle \eta \cdot
				\varphi_ ,\nabla _{X}\varphi_j \right\rangle \right\rangle ^{\bot } +\sum_{i,j=1}^{2^k}\left\langle \left\langle
				\nabla _{X}\eta \cdot \varphi_i ,\varphi_j \right\rangle \right\rangle ^{\bot
				}.
				\end{eqnarray*}
			\textbf{Statement:}
				\begin{equation*}
				\sum_{i,j=1}^{2^k}\left\langle \left\langle \eta \cdot \nabla _{X}\varphi_i ,\varphi_j
				\right\rangle \right\rangle ^{\bot }+\sum_{i,j=1}^{2^k}\left\langle \left\langle \eta \cdot
				\varphi_i ,\nabla _{X}\varphi_j \right\rangle \right\rangle ^{\bot }=0. 
				\end{equation*}
			Indeed,
				\begin{multline*}
				\sum_{i,j=1}^{2^k}\left\langle \left\langle \eta \cdot \nabla _{X}\varphi _i,\varphi_j
				\right\rangle \right\rangle +\sum_{i,j=1}^{2^k}\left\langle \left\langle \eta \cdot \varphi_i
				,\nabla _{X}\varphi_j \right\rangle \right\rangle =\sum_{i,j=1}^{2^k}(id-\tau )\left\langle \left\langle \eta \cdot \nabla _{X}\varphi_i
				,\varphi_j \right\rangle \right\rangle   \\	\shoveleft{=\sum_{i,j=1}^{2^k}(-id+\tau )\left\langle \left\langle \left[ \frac{1}{2}
				\sum _{r=1}^{p}\eta \cdot e_{r}\cdot B(X,e_{r})\cdot \varphi_i  -\frac{1}{2}\mathbf{i}~A^{l}(X)\eta \cdot \varphi_i \right] ,\varphi_j \right\rangle \right\rangle }\\	\shoveleft{=\sum_{i,j=1}^{2^k}(-id+\tau )\left\langle \left\langle \left[ -\frac{1}{2}
					\sum _{r=1}^{p}\sum _{s=1}^{q}\sum
				 _{t=1}^{q}a^{s}b_{r}^{t}e_{r}\cdot f_{s}\cdot f_{t} -\frac{1}{2}
					\mathbf{i}~A^{l}(X)\eta \right] \cdot \varphi_i ,\varphi _j \right\rangle \right\rangle }
					\\	\shoveleft{=\sum_{i,j=1}^{2^k}(-id+\tau )\left\langle \left\langle \left[ \frac{1}{2}
					\sum _{r=1}^{p}\sum _{s=1}^{q} a^{s}b_{r}^{s}e_{r} \right. \right. \right.  }
					\\ \left. \left. \left. -\frac{1}{2}\sum _{r=1}^{p}\sum _{s=1}^{q}\sum _{t=1,t\neq
					s}^{q} a^{s}b_{r}^{t}e_{r}\cdot f_{s}\cdot f_{t}-\frac{1}{2}A^{l}(X)i \eta \right]
					\cdot \varphi _i,\varphi_j \right\rangle \right\rangle ,
				\end{multline*}
			from which it follows that
				\begin{eqnarray*}
					&&\sum_{i,j=1}^{2^k}\left\langle \left\langle \eta \cdot \nabla _{X}\varphi_i ,\varphi_j \right\rangle
					\right\rangle +\sum_{i,j=1}^{2^k}\left\langle \left\langle \eta \cdot \varphi_i ,\nabla _{X}\varphi_j
					\right\rangle \right\rangle \\ && = \sum_{j=1}^{2^k}\tau \lbrack \varphi_j ]\left[ \frac{1}{2}\sum_{r=1}^{p}	\sum_{s=1}^{q}a^{s}b_{r}^{s}e_{r}\right]\sum_{i=1}^{2^k}[\varphi_i ] 
					+\sum_{j=1}^{2^k}\tau \lbrack \varphi_j ]\left[\frac{1}{2}\sum_{r=1}^{p}\sum_{s=1}^{q}a^{s}b_{r}^{s}e_{r}\right]\sum_{i=1}^{2^k}[\varphi_i ]  \\
					&&= \tau \lbrack \varphi ]\left[\frac{1}{2}\sum_{r=1}^{p}\sum_{s=1}^{q}a^{s}b_{r}^{s}e_{r}\right][\varphi ]+\tau \lbrack \varphi ]\left[\frac{1}{2}\sum_{r=1}^{p}\sum_{s=1}^{q}a^{s}b_{r}^{s}e_{r}\right][\varphi ]  \\
					&&=\tau \lbrack \varphi]\left[\frac{1}{2}\sum_{r=1}^{p}\sum_{s=1}^{q}a^{s}b_{r}^{s}e_{r}\right][\varphi ]=\tau \lbrack
					\varphi ][\nu][\varphi ]=:\xi(\nu)\in TF(M) \notag \\
					&\Rightarrow &\sum_{i,j=1}^{2^k}\left\langle \left\langle \eta \cdot \nabla _{X}\varphi_i
					,\varphi_j \right\rangle \right\rangle ^{\bot }+\sum_{i,j=1}^{2^k}\left\langle \left\langle \eta
					\cdot \varphi_i ,\nabla _{X}\varphi_j \right\rangle \right\rangle ^{\bot }=0.
				\end{eqnarray*}
				Concluding
				\begin{equation*}
				\nabla _{\xi (X)}^{F}\xi (\eta )=\sum_{i,j=1}^{2^k}\left\langle \left\langle \nabla _{X}\eta
				\cdot \varphi_i ,\varphi_j \right\rangle \right\rangle ^{\bot }=\Big(\xi (\nabla _{X}\eta )\Big)^{\bot }=\xi (\nabla
				_{X}^{\prime }\eta ). 
				\end{equation*}%
			Finally $ \mathbf{ii)} $ follows.
			\end{enumerate}
		\end{proof}

Having established this, we will have the following:

\begin{theorem} \label{principal}
	Let $M$ a riemannian $p$-dimensional manifold, $E\rightarrow M$ a real vector bundle of rank $q$, assume that $TM$ and $E$ are oriented and $Spin^{\mathbb{C}}.$ Suppose that $B:TM\times TM\rightarrow E$ is a bilinear and symmetric form . Thus the following statements are equivalent:
	
	\begin{enumerate}
		\item 
		\textbf{For the case } $n = 2k$ \textbf{ even:}
		
		There are $2^{k}$ spinors $\varphi _{i}\in \Gamma \left( \Sigma _{1}^{\mathbb{C}}\right) \simeq \Gamma \left( \Sigma _{i}^{\mathbb{C}}\right) $
		which satisfy the equations
		\begin{equation*}
		\nabla _{X}^{\Sigma _{1}^{\mathbb{C}}}\varphi _{i}=-\frac{1}{2}%
		\sum_{j=1}^{p}e_{j}\cdot B(X,e_{j})\cdot \varphi _{i}+\frac{1}{2}%
		\textbf{i}~A^{l}(X)\cdot \varphi _{i},i=1,\cdots ,2^{k}.  \label{killingiipar}
		\end{equation*}
		
		\textbf{For the case } $n = 2k+1$ \textbf{ odd:}
		
		There are $2^{k+1}$ spinors $\varphi _{i;0}\in \Gamma \left( \Sigma _{1;0}^{\mathbb{C}}\right) $, $\varphi _{i;1}\in \Gamma \left( \Sigma _{1;1}^{\mathbb{C}}\right) $
		which satisfy the equations
		\begin{equation*}
        \nabla _{X}^{\Sigma _{1;\lambda}^{\mathbb{C}}}\varphi _{i;\lambda}=-\frac{1}{2}%
		\sum_{j=1}^{p}e_{j}\cdot B(X,e_{j})\cdot \varphi _{i;\lambda}+\frac{1}{2}%
		\textbf{i}~A^{l}(X)\cdot \varphi _{i;\lambda}, \ \ i=1,\cdots ,2^{k}, \lambda=0,1. \label{killingiiimpar}
		\end{equation*}

		\item There is an isometric immersion $F:M\rightarrow \mathbb{R}%
		^{\left( n+m\right) }$ with normal bundle $E$ and second fundamental form $B$.
	\end{enumerate}
	
	Besides that, $dF= \xi $ where $\xi $ is a $\mathbb{R}^{n}$-valued $1$-form defined by
	
	\textbf{For the case } $n= 2k$ \textbf{ even:}
	\begin{eqnarray*}
	\xi (X) &=& \sum_{i,j}^{2^k} \xi_{ij} (X),  \label{formpar} \\
\xi _{ij}(X) &=&\left\langle \left\langle X\cdot \varphi _{i},\varphi
_{j}\right\rangle \right\rangle ,i,j=1,\cdots ,2^{k},~~\forall X\in TM.
	\end{eqnarray*}
	
	\textbf{For the case } $n = 2k+1$ \textbf{odd:}
	\begin{eqnarray*}
\xi (X) &=& \sum_{\lambda=0}^{1} \sum_{i,j}^{2^k}  \xi _{ij;\lambda}(X) \label{formimpar} \\
	\xi _{ij;\lambda}(X) &=&\left\langle \left\langle X\cdot \varphi _{i;\lambda},\varphi
	_{j;\lambda}\right\rangle \right\rangle ,i,j=1,\cdots ,2^{k}, \lambda=0,1,~~\forall X\in TM.
	\end{eqnarray*}
\end{theorem}

\begin{proof}
	The proof immediately follows from lemmas \ref{spinoresidealcasopar}, \ref{spinoresidealcasoimpar}, \ref{lemaxii}, \ref{lemaxii2}, \ref{lemaimersaoo}.
\end{proof}


\appendix

\section{Complex Clifford algebra ideal spinors of an immersed manifold} \label{apendicea}

Using the irreducible representations of complex Clifford algebras, we define the following bundles of complex spinors:

\textbf{Case }$p$ \textbf{and }$q$ \textbf{even} 
\begin{align*}
&\sum\nolimits_{i}^{\mathbb{C}}M :=P_{Spin^{\mathbb{C}}_p}\times _{\rho
_{i}^{p}}I_{i}^{p},~\sum\nolimits_{j}^{\mathbb{C}}E:=P_{Spin^{\mathbb{C}}_q}\times _{\rho _{j}^{q}}I_{i}^{q},  \\
&\sum\nolimits_{r}^{\mathbb{C}}Q :=P_{Spin^{\mathbb{C}}_n}\times _{\rho
_{r}^{n}}I_{r}^{n},~\left. \sum\nolimits_{r}^{\mathbb{C}%
}Q\right\vert _{M}:=\left. P_{Spin^{\mathbb{C}}_n}\right\vert _{M}\times
_{\rho _{r}^{n}}I_{r}^{n},  \\
&i =1,\cdots ,2^{\frac{p}{2}};j=1,\cdots ,2^{\frac{q}{2}};r=1,\cdots ,2^{\frac{n}{2}};
\end{align*}

\textbf{Case }$p$ \textbf{even, }$q$ \textbf{odd} 
\begin{align*}
&\sum\nolimits_{i}^{\mathbb{C}}M :=P_{Spin^{\mathbb{C}}_p}\times _{\rho
_{i}^{p}}I_{i}^{p},~\sum\nolimits_{j;\lambda}^{\mathbb{C}}E:=P_{Spin^{\mathbb{C}}_q}\times _{\rho _{j;\lambda}^{q}}I_{j;\lambda}^{q},   \\ & \sum\nolimits_{r;\lambda}^{\mathbb{C}}Q :=P_{Spin^{\mathbb{C}}_n}\times
_{\rho _{r;\lambda}^{n}}I_{r;\lambda}^{n}, \left. \sum\nolimits_{r;\lambda}^{\mathbb{C}}Q\right\vert _{M}:=\left. P_{Spin^{\mathbb{C}}_n}\right\vert _{M}\times
_{\rho _{r;\lambda}^{n}}I_{r;\lambda}^{n},  \\
&  ~ i =1,\cdots ,2^{\frac{p}{2}};j=1,\cdots ,2^{\frac{q-1}{2}};r=1,\cdots ,2^{%
\frac{n-1}{2}};\lambda=0,1.
\end{align*}

\textbf{Case }$p$ \textbf{odd, }$q$ \textbf{even} 
\begin{align*}
&\sum\nolimits_{i;\lambda}^{\mathbb{C}}M :=P_{Spin^{\mathbb{C}}_p}\times _{\rho
_{i;\lambda}^{p}}I_{i;\lambda}^{p},~\sum\nolimits_{j}^{\mathbb{C}}E:=P_{Spin^{\mathbb{C}}_q}\times _{\rho _{j}^{q}}I_{j}^{q},  \\
& \sum\nolimits_{r;\lambda}^{\mathbb{C}}Q :=P_{Spin^{\mathbb{C}}_n}\times
_{\rho _{r;\lambda}^{n}}I_{r;\lambda}^{n},~\left. \sum\nolimits_{r;\lambda}^{\mathbb{C}}Q\right\vert _{M}:=\left. P_{Spin^{\mathbb{C}}_n}\right\vert _{M}\times
_{\rho _{r;\lambda}^{n}}I_{r,\lambda}^{n},  \\
&i =1,\cdots ,2^{\frac{p-1}{2}};j=1,\cdots ,2^{\frac{q}{2}};r=1,\cdots ,2^{%
\frac{n-1}{2}};\lambda=0,1.
\end{align*}

\textbf{Case }$p$ \textbf{and }$q$ \textbf{even} 
\begin{align*}
&\sum\nolimits_{i;\lambda}^{\mathbb{C}}M :=P_{Spin^{\mathbb{C}}_p}\times _{\rho_{i}^{p}}I_{i;\lambda}^{p},~\sum\nolimits_{j;\lambda^{\prime }}^{\mathbb{C}}E:=P_{Spin^{\mathbb{C}}m}\times _{\rho _{j;\lambda^{\prime }}^{q}}I_{j;\lambda^{\prime }}^{q},  \\
&\sum\nolimits_{r}^{\mathbb{C}}Q :=P_{Spin^{\mathbb{C}}_n}\times _{\rho_{r}^{n}}I_{r}^{n},~\left. \sum\nolimits_{r}^{\mathbb{C}}Q\right\vert _{M}:=\left. P_{Spin^{\mathbb{C}}_n}\right\vert _{M}\times
_{\rho _{r}^{n}}I_{r}^{n},  \\
&i =1,\cdots ,2^{\frac{p-1}{2}};j=1,\cdots ,2^{\frac{q-1}{2}};r=1,\cdots
,2^{\frac{n}{2}};\lambda,l^{\prime }=0,1.
\end{align*}

From \cite[pp. 5]{bar98} we can compare the spinor modules of the Clifford algebra of a direct sum with the spinor modules associated with each factor. Besides that, using the fact that the transition functions of $\left.P_{Spin^{\mathbb{C}}_n}\right\vert_{M}$ are the product of the transition functions of $P_{Spin^{\mathbb{C}}p}$ and $P_{Spin^{\mathbb{C}}q}$ , it's not difficult to get:

\textbf{Case }$p$ \textbf{and }$q$ \textbf{even:} 
\begin{eqnarray} \label{app1}
\sum\nolimits_{i;j}^{\mathbb{C}} &:=&\sum\nolimits_{i}^{\mathbb{C}}M\otimes \sum\nolimits_{j}^{\mathbb{C}}N=\left( P_{Spin^{\mathbb{C}}_p}\times_{\rho _{i}^{p}}I_{i}^{p}\right) \otimes \left( P_{Spin^{\mathbb{C}}_q}\times _{\rho _{j}^{q}}I_{j}^{q}\right) \notag  \\
&\simeq &\left. P_{Spin^{\mathbb{C}}_n}\right\vert _{M}\times _{\rho
}\left( I_{i}^{p}\otimes I_{j}^{q}\right) 
\simeq \left. P_{Spin^{\mathbb{C}}_n}\right\vert _{M}\times _{\rho_{r}^{n}}I_{r}^{n}=:\left. \sum\nolimits_{r}^{\mathbb{C}}Q\right\vert _{M} \\
&&\forall i=1,\cdots ,2^{\frac{p}{2}};j=1,\cdots ,2^{\frac{q}{2}};r=1,\cdots
,2^{\frac{n}{2}}. \notag
\end{eqnarray}

\textbf{Case }$p$ \textbf{even, }$q$ \textbf{odd:} 
\begin{eqnarray} \label{app2}
\sum\nolimits_{i,j;\lambda}^{\mathbb{C}} &:=&\sum\nolimits_{i}^{\mathbb{C}}M\otimes \sum\nolimits_{j;\lambda}^{\mathbb{C}}N=\left( P_{Spin^{\mathbb{C}}_p}\times _{\rho _{i}^{p}}I_{i}^{p}\right) \otimes \left( P_{Spin^{\mathbb{C}}_q}\times _{\rho _{j;\lambda}^{q}}I_{j;\lambda}^{q}\right) \notag \\
&\simeq &\left. P_{Spin^{\mathbb{C}}_n}\right\vert _{M}\times _{\rho}\left( I_{i}^{p}\otimes I_{j;\lambda}^{q}\right) 
\simeq \left. P_{Spin^{\mathbb{C}}_n}\right\vert _{M}\times _{\rho
_{r;\lambda}}I_{r;\lambda}^{n}=:\left. \sum\nolimits_{r;\lambda}^{\mathbb{C}}Q\right\vert_{M}
\\
&& \forall i=1,\cdots ,2^{\frac{p}{2}};j=1,\cdots ,2^{\frac{q-1}{2}};r=1,\cdots ,2^{\frac{n-1}{2}};\lambda=0,1. \notag
\end{eqnarray}

\textbf{Case }$p$ \textbf{odd, }$q$ \textbf{even:} 
\begin{eqnarray} \label{app3}
\sum\nolimits_{i;\lambda,j}^{\mathbb{C}}&:=&\sum\nolimits_{i;\lambda}^{\mathbb{C}}M\otimes \sum\nolimits_{j}^{\mathbb{C}}N=\left( P_{Spin^{\mathbb{C}}_p}\times _{\rho _{i;\lambda}^{p}}I_{i;\lambda}^{p}\right) \otimes \left( P_{Spin^{\mathbb{C}}_q}\times _{\rho _{j}^{q}}I_{j}^{q}\right)  \notag \\
&\simeq &\left. P_{Spin^{\mathbb{C}}_n}\right\vert _{M}\times _{\rho}\left( I_{i;\lambda}^{p}\otimes I_{j}^{q}\right)  
\simeq \left. P_{Spin^{\mathbb{C}}_n}\right\vert _{M}\times _{\rho
_{r;\lambda}}I_{r;\lambda}^{n}=:\left. \sum\nolimits_{r;\lambda}^{\mathbb{C}}Q\right\vert_{M}\\ && \forall i=1,\cdots ,2^{\frac{p-1}{2}};j=1,\cdots ,2^{\frac{q}{2}};r=1,\cdots ,2^{\frac{n-1}{2}};\lambda=0,1. \notag
\end{eqnarray}

\textbf{Case }$p$ \textbf{and }$q$ \textbf{odd:} 
\begin{eqnarray} \label{app4}
\sum\nolimits_{i,j}^{\mathbb{C}} &:=&\left( \sum\nolimits_{i;0}^{\mathbb{C}}M\otimes \sum\nolimits_{j;0}^{\mathbb{C}}N\right) \oplus \left(
\sum\nolimits_{i;0}^{\mathbb{C}}M\otimes \sum\nolimits_{j;1}^{\mathbb{C}}N\right)  \notag \\
&=&\left( P_{Spin^{\mathbb{C}}_p}\times _{\rho _{i;0}^{p}}I_{i;0}^{p}\right)
\otimes \left( P_{Spin^{\mathbb{C}}_q}\times _{\rho
_{j;0}^{q}}I_{j;0}^{q}\right)  \oplus \left( P_{Spin^{\mathbb{C}}_p}\times
_{\rho _{i;0}^{p}}I_{i;0}^{p}\right) \otimes \left( P_{Spin^{\mathbb{C}}_q}\times _{\rho _{j;1}^{q}}I_{j;1}^{q}\right)  \notag \\
&\simeq &\left. P_{Spin^{\mathbb{C}}_n}\right\vert _{M}\times _{\rho
}\left( I_{i;0}^{p}\otimes I_{j;0}^{q}\right) \oplus \left(
I_{i;0}^{p}\otimes I_{j;1}^{q}\right) 
\simeq \left. P_{Spin^{\mathbb{C}}_n}\right\vert _{M}\times
_{r}I_{r}^{n}=:\left. \sum\nolimits_{r}^{\mathbb{C}}Q\right\vert _{M} \\ && \forall i=1,\cdots ,2^{\frac{p-1}{2}};j=1,\cdots ,2^{\frac{q-1}{2}};r=1,\cdots ,2^{\frac{n}{2}}. \notag
\end{eqnarray}

In eqs. \eqref{app1}, \eqref{app2}, \eqref{app3} and \eqref{app4}, $\rho$ is the representation given in \cite[pp. 5]{bar98}.


%
%

%



\end{document}